\documentclass{amsart}

\usepackage{leftidx,multicol,amssymb}

\newtheorem{thm}{Theorem}
\newtheorem{prop}[thm]{Proposition} 
\newtheorem{lem}[thm]{Lemma}
\newtheorem{cor}[thm]{Corollary}

\newtheorem*{claim}{Claim}

\theoremstyle{remark}
\newtheorem*{rem}{Remark}

\theoremstyle{definition}
\newtheorem{defn}[thm]{Definition}

\newtheorem{example}[thm]{Example} 
 
\newtheorem{question}[thm]{Question}
\newtheorem*{axm}{Axiom}

\title[From alpha-Ramsey to ultra-Ramsey]{From abstract alpha-Ramsey theory to abstract ultra-Ramsey theory}

\author{Timothy Trujillo}

\address{Colorado School of Mines\\ Department of Applied Mathematics and Statistics\\ 1500 Illinois St., Golden, CO 80401, USA}

\email{trujillo@mines.edu}

\begin{document}
\maketitle

\begin{abstract}
We work within the framework of the Alpha-Theory introduced by Benci and Di Nasso. The Alpha-Theory postulates a few natural properties for an infinite ``ideal'' number $\alpha$. The formulation provides an elementary axiomatics for the methods of abstract ultra-Ramsey theory. 

The main results are Theorem \ref{alpha Ramsey theorem}, Theorem \ref{abstract alpha Ramsey theorem}, Theorem \ref{abstract alpha-Ellentuck theorem} and Theorem \ref{abstract ultra-Ellentuck theorem}. Theorem \ref{alpha Ramsey theorem} is an infinite-dimensional extension of the celebrated Ramsey's Theorem. We show that corollaries of this result include the Galvin-Pirky Theorem, the Silver Theorem and the $\vec{\alpha}$-Ellentuck Theorem. We prove that, under the assumption of the $\mathfrak{c}^{+}$-enlarging property, the $\vec{\alpha}$-Ellentuck Theorem is equivalent to the Ultra-Ellentuck Theorem of Todorcevic. Moreover, we use the results to prove a theorem of Louveau about the infinite-dimensional Ramsey theory of selective ultrafilters, and provide a new Ramsey theoretic characterization of the strong Cauchy infinitesimal principle introduced by Benci and Di Nasso.

Theorem \ref{abstract alpha Ramsey theorem} is an abstraction of Theorem \ref{alpha Ramsey theorem} to the setting of triples $(\mathcal{R},\le,r)$ where $\mathcal{R}\not=\emptyset$, $\le$ is a quasi-order on $\mathcal{R}$ and $r$ is a function with domain $\mathbb{N}\times \mathcal{R}$. We use Theorem \ref{abstract alpha Ramsey theorem} to develop the Abstract $\vec{\alpha}$-Ellentuck Theorem, Theorem \ref{abstract alpha-Ellentuck theorem}, and the Abstract Ultra-Ellentuck Theorem, Theorem \ref{abstract ultra-Ellentuck theorem}, extending the Abstract Ellentuck Theorem along the same lines as the $\vec{\alpha}$-Ellentuck Theorem and Ultra-Ellentuck Theorem extend the Ellentuck Theorem, respectively. We conclude with some examples illustrating the theory and an open question related to the local Ramsey theory developed by Di Prisco, Mijares and Nieto.
\end{abstract}

\tableofcontents

\section{Introduction}
All of the main theorems and corollaries follow from Theorem \ref{alpha Ramsey theorem} and Theorem \ref{abstract alpha Ramsey theorem}. These two theorems can be seen as unrestricted infinite-dimensional generalizations of Ramsey's Theorem in the setting of $\vec{\alpha}$-trees. Here $\vec{\alpha}$ denotes some sequence of nonstandard hypernatural numbers indexed by the collection of finite subsets of $\mathbb{N}$. In fact, we show that Ramsey's Theorem follows directly from Theorem \ref{alpha Ramsey theorem} and an abstract form of Ramsey's Theorem follows from Theorem \ref{abstract alpha Ramsey theorem}. 

The Alpha-Theory introduced by Benci and Di Nasso in \cite{AlphaTheory} provides an elementary theoretical foundation for studying $\vec{\alpha}$-trees and developing the $\vec{\alpha}$-Ramsey theory. Under certain saturation assumptions on the Alpha-Theory, $\vec{\alpha}$-trees coincide with $\vec{\mathcal{U}}$-trees as introduced by Blass in \cite{BlassU-trees}. The simplicity of the Alpha-Theory makes the proofs of Theorem \ref{alpha Ramsey theorem} and Theorem \ref{abstract alpha Ramsey theorem} readily apparent. Theorem \ref{alpha Ramsey theorem} does not appear in \cite{BlassU-trees} where $\vec{\mathcal{U}}$-trees are introduced nor in \cite{RamseySpaces} where the ultra-Ramsey theory is developed.

\section{The Alpha-Theory}
Nonstandard analysis was introduced by Robinson in \cite{Robinson1, Robinson2} to reintroduced infinitesimal and infinite numbers into analysis. More recently, Di Nasso and Baglini have had success applying nonstandard analysis to Ramsey theory see \cite{NSA1,NSA2,NSA3,NSA4,NSA5}. Using model theory, Robinson gave a rigorous development of the calculus of infinitesimals. Unfortunately, for many researchers the formalism appeared to be too technical. In an analogous way the ultra-Ramsey theory also can be seen as too technical. Recently, Benci and Di Nasso in \cite{AlphaTheory} have introduced a simplified presentation of nonstandard analysis called the Alpha-Theory. Their presentation shows that technical concepts such as ultrafilter, ultrapower, superstructure and the $*$-transfer principle are not needed to rigorously develop calculus with infinitesimals. In this paper, we show that the same elementary foundation can be used to develop ultra-Ramsey theory and abstract ultra-Ramsey theory. 

Benci and Di Nasso in \cite{AlphaTheory} describe the Alpha-Theory as ``an axiomatic system that postulates a few natural properties for an infinite ``ideal'' natural number $\alpha$.'' The idea of adjoining a new number that behaves like a very large natural number goes back to work of Schmeiden and Laugwitz \cite{omegatheory}. They adjoin a new symbol $\Omega$ and assume that a `formula' is true at $\Omega$ if it is true for all sufficiently large natural numbers. Benci and Di Nasso in \cite{AlphaTheory} state that the Alpha-Theory approach can be seen as a strengthening of the $\Omega$-Theory introduced by Schmeiden and Laugwitz in \cite{omegatheory}. In this section, we follow \cite{AlphaTheory} and give an informal presentation of the Alpha-Theory. For a formal presentation of the Alpha-Theory as a first-order theory see the final section of \cite{AlphaTheory}.

Before introducing the axioms of the Alpha-Theory we make the assumption, as in \cite{AlphaTheory}, that all usual axioms of ZFC are true. We introduce a new symbol $\alpha$ whose properties are postulated by the following five axioms.

\begin{axm}[$\alpha$1 Extension]
For all sequences $\left<\varphi_{i} : i\in \mathbb{N}\right>$ there a unique element $\varphi[\alpha]$, called the ``ideal value of $\varphi$.''
\end{axm}
\begin{axm}[$\alpha$2 Composition]
If $\left<\varphi_{i} : i\in \mathbb{N}\right>$ and $\left<\psi_{i} : i\in \mathbb{N}\right>$ are sequences and $f$ is any function such that $f\circ \varphi$ and $f\circ \psi$ make sense, then 
$$\varphi[\alpha]=\psi[\alpha] \implies (f\circ \varphi)[\alpha] = (f\circ \psi)[\alpha].$$
\end{axm}
\begin{axm}[$\alpha$3 Number]
Suppose that $r\in \mathbb{R}$. Let $\left<\varphi_{i} : i\in\mathbb{N}\right>$ and $\left<\psi_{i}:i\in \mathbb{N}\right>$ be the sequences such that for all $i\in \mathbb{N}$, $\varphi_{i}=r$ and $\psi_{i} = i$. Then $\varphi[\alpha]=r$ and $\psi[\alpha]= \alpha\not \in \mathbb{N}$.
\end{axm}

\begin{axm}[$\alpha$4 Pair]
For all sequences $\left<\varphi_{i} : i\in \mathbb{N}\right>$, $\left<\psi_{i} : i\in \mathbb{N}\right>$ and $\left<\vartheta_{i} : i\in \mathbb{N}\right>$,
$$(\forall i\in \mathbb{N},\ \vartheta_{i} =\{\varphi_{i},\psi_{i}\}) \implies \vartheta[\alpha] =\{\varphi[\alpha],\psi[\alpha]\}.$$
\end{axm}

\begin{axm}[$\alpha$5 Internal Set]
If $\left<\varphi_{i} : i\in \mathbb{N}\right>$ is the sequence such that for all $i\in \mathbb{N}$, $\varphi_{i} =\emptyset$, then $\varphi[\alpha]=\emptyset$. If $\left<\psi_{i} : i\in \mathbb{N}\right>$ is a sequence of nonempty sets, then 
$$\psi[\alpha] =\{ \vartheta[\alpha] : \forall i\in \mathbb{N}, \vartheta_{i} \in \psi_{i}\}.$$
\end{axm}

\begin{defn}
For all sets $A$, we let $\leftidx{^{*}}{A}$ denote the ideal value of the constant sequence $\left<\varphi_{i}:i\in\mathbb{N}\right>$ where for all $i\in\mathbb{N}$, $\varphi_{i}= A$. We call $\leftidx{^{*}}{A}$ the \emph{$*$-transform of $A$}. Note that by the Axiom $\alpha 5$, $\leftidx{^{*}}{A}$ consists of the ideal values of sequences of elements from $A$.
\end{defn}
The set of \emph{hypernatural numbers} is the $*$-transform of the set of natural numbers. Notice that by the number axiom and the internal set axiom, $\alpha \in \leftidx{^{*}}{\mathbb{N}}\setminus \mathbb{N}.$ The set of \emph{nonstandard hypernatural numbers} is exactly the set $\leftidx{^{*}}{\mathbb{N}}\setminus \mathbb{N}.$ In particular, $\alpha$ is an example of a nonstandard hypernatural number. The next proposition follows easily from $\alpha 1$-$\alpha 5$ (for a proof see \cite{AlphaTheory}). The proposition shows that the $*$-transform preserves all basic operations of sets with the exception of the powerset.

\begin{prop}[Proposition 2.2, \cite{AlphaTheory}]\label{star transform} For all sets $A$ and $B$ the following hold:
\begin{multicols}{2}
\begin{enumerate}
\item $A=B \iff \leftidx{^{*}}{A} = \leftidx{^{*}}{B}$
\item $A\in B \iff \leftidx{^{*}}{A} \in \leftidx{^{*}}{B}$
\item $A \subseteq B \iff \leftidx{^{*}}{A} \subseteq \leftidx{^{*}}{B}$
\item $\leftidx{^{*}}{\{A,B\}} = \{\leftidx{^{*}}{A},\leftidx{^{*}}{B}\}$
\item $\leftidx{^{*}}{(A,B)} = (\leftidx{^{*}}{A},\leftidx{^{*}}{B})$
\item $\leftidx{^{*}}{(A\cup B)} = \leftidx{^{*}}{A}\cup\leftidx{^{*}}{B}$
\item $\leftidx{^{*}}{(A\cap B)} = \leftidx{^{*}}{A}\cap\leftidx{^{*}}{B}$
\item $\leftidx{^{*}}{(A\setminus B)} = \leftidx{^{*}}{A}\setminus\leftidx{^{*}}{B}$
\item $\leftidx{^{*}}{(A\times B)} = \leftidx{^{*}}{A}\times\leftidx{^{*}}{B}$
\end{enumerate}
\end{multicols}
\end{prop}

 Recall that in ZFC a binary relation $R$ between two sets $A$ and $B$ is identified with the set $\{(x,y)\in A\times B : x R y\}$. Hence, $\leftidx{^{*}}{R}$ is a binary relation between $\leftidx{^{*}}{A}$ and $\leftidx{^{*}}{B}$. The same holds of $n$-place relations. In particular, if $f:A\rightarrow B$ is a function then $\leftidx{^{*}}{f}$ is a binary relation between $\leftidx{^{*}}{A}$ and $\leftidx{^{*}}{B}$. The next proposition shows that $\leftidx{^{*}}{f}$ is also a function.

\begin{prop}[Proposition 2.3, \cite{AlphaTheory}]
Let $f:A\rightarrow B$ be a function. Then $\leftidx{^{*}}{f}:\leftidx{^{*}}{A} \rightarrow \leftidx{^{*}}{B}$ is a function such that, for every sequence $\varphi:\mathbb{N}\rightarrow A$,
$$\leftidx{^{*}}{f}(\varphi[\alpha]) = (f\circ \varphi)[\alpha].$$
Moreover, $f$ is one-to-one if and only if $\leftidx{^{*}}{f}$ is one-to-one.
\end{prop}

\begin{prop}\label{infinite set}
Suppose that $\beta$ is a nonstandard hypernatural number and $X\subseteq \mathbb{N}$. If $\beta\in\leftidx{^{*}}{X}$ then $X$ is infinite.
\end{prop}
\begin{proof} We prove the contrapostive. That is, if $X$ is finite then $\beta\not \in \leftidx{^{*}}{X}$. Suppose that $X=\{x_{0}, x_{1},\dots, x_{n}\}$ is a finite subset of the natural numbers. By $\alpha 3$ and $\alpha 4$, $\leftidx{^{*}}{\{x_{0}, x_{1},\dots, x_{n}\}} = \{\leftidx{^{*}}{x_{0}}, \leftidx{^{*}}{x_{1}},\dots, \leftidx{^{*}}{x_{n}}\} =\{x_{0}, x_{1},\dots, x_{n}\} \subseteq \mathbb{N}$. Since $\beta\not \in \mathbb{N}$, $\beta\not \in \leftidx{^{*}}{X}$. 
\end{proof}

One of the fundamental tools of nonstandard analysis is the use of saturation principles. Benci and Di Nasso in \cite{AlphaTheory} state that ``the Alpha-Theory can be  generalized so to accommodate all nonstandard arguments which use a prescribed level of saturation." Later we show that, under the assumption of the $\mathfrak{c}^{+}$-enlarging property (a saturation principle), the $\vec{\alpha}$-Ellentuck Theorem is equivalent to the Ultra-Ellentuck Theorem of Todorcevic. The $\mathfrak{c}^{+}$-enlarging property is not a theorem of the Alpha-Theory; however, the countable enlarging property does follow from $\alpha$1-$\alpha$5. We omit its proof as it follows by a direct application of Theorem 4.4 in \cite{AlphaTheory}.
\begin{prop}[Countable enlarging property] 
Suppose $\{A_{i}:i\in\mathbb{N}\}$ is a countable family of subsets of some set $A$ with the finite intersection property, i.e. such that any finite intersection $A_{0}\cap\cdots\cap A_{n}\not=\emptyset$. Then $$\bigcap_{i=0}^{\infty}\leftidx{^{*}}{A_{i}} \not=\emptyset.$$
\end{prop}

Throughout the remainder of this article we will use the propositions of this section implicitly. In order to keep the proofs less cumbersome we only explicitly quote these results when confusion may arise. We also follow the practice in nonstandard analysis, when confusion is unlikely, of dropping the $*$ symbol from $*$-transforms. For example, if $\beta$ is nonstandard hypernatural number then we write $\forall n\in \mathbb{N},$ $n< \beta$, instead of $\forall n\in \mathbb{N}$, $n \ \leftidx{^{*}}{<} \beta.$

\section{Alpha-Ramsey Theory}
We fix the notation we will use for the remainder of the paper regarding subsets of the natural numbers. For $n\in \mathbb{N}$ and $X\subseteq \mathbb{N}$, we use the following:
$$[X]^{n} = \{ Y\subseteq X : |Y| = n \},$$
$$[X]^{<\infty} = \{ Y\subseteq X : |Y| <\infty \},$$
$$[X]^{\infty} = \{ Y\subseteq X : |Y| = \infty \}.$$
If $s\in[\mathbb{N}]^{<\infty}$ and $X\subseteq \mathbb{N}$ then we say \emph{$s$ is an initial segment of $X$} and write $s\sqsubseteq X$, if there exists $i\in\mathbb{N}$ such that $s = \{ j\in X : j\le i\}$. If $s\sqsubseteq X$ and $s\not=X$ then we write $s\sqsubset X$. 

\begin{defn} 
A subset $T$ of $[\mathbb{N}]^{<\infty}$ is called a \emph{tree on $\mathbb{N}$} if $T\not =\emptyset$ and for all $s,t\in[\mathbb{N}]^{<\infty}$,
$$s\sqsubseteq t\in T \implies s\in T.$$ For a tree $T$ on $\mathbb{N}$ and $n\in \mathbb{N}$, we use the following notation: 
$$[T] =\{X\in[\mathbb{N}]^{\infty} : \forall s\in [\mathbb{N}]^{<\infty}( s\sqsubseteq X \implies s \in T) \},$$
$$T(n) =\{s\in T : |s|=n\}.$$
The \emph{stem of $T$}, if it exists, is the $\sqsubseteq$-maximal $s$ in  $T$ that is $\sqsubseteq$-comparable to every element of $T$. If $T$ has a stem we denote it by $st(T)$. For $s\in T$, we use the following notation $$T/s = \{ t\in T : s\sqsubseteq t\}.$$
\end{defn}

For the remainder of this section we fix a sequence $\vec{\alpha}=\left< \alpha_{s} : s\in [\mathbb{N}]^{<\infty}\right>$ where each $\alpha_{s}$ is a nonstandard hypernatural number. Note that in the Alpha-Theory at least one such sequence exists, $\alpha 3$ and $\alpha 5$ imply that $\alpha\in \leftidx{^{*}}{\mathbb{N}}\setminus \mathbb{N}$, take $\vec{\alpha}$ to be the sequence where $\alpha_{s} =\alpha$ for all $s\in[\mathbb{N}]^{<\infty}$. 

\begin{defn}
An \emph{$\vec{\alpha}$-tree} is a tree $T$ with stem $st(T)$ such that $T/st(T)\not=\emptyset$ and for all $s\in T/st(T)$, $$s\cup\{\alpha_{s}\} \in  \leftidx{^{*}}{T}.$$
\end{defn}

\begin{example}Note that $[\mathbb{N}]^{<\infty}$ is a tree on $\mathbb{N}$ with stem $\emptyset$. Moreover, for all $s\in[\mathbb{N}]^{\infty}$, $s\cup\{\alpha_{s}\}\in \leftidx{^{*}}{[\mathbb{N}]^{<\infty}}$. Thus, $[\mathbb{N}]^{<\infty}$ is an $\vec{\alpha}$-tree.
\end{example}
The proof of the next lemma is nearly identical to the proof of Lemma 7.33 of Todorcevic in \cite{RamseySpaces}. The only difference is that we use $\vec{\alpha}$-trees instead of the ultrafilter trees used in \cite{RamseySpaces}.
\begin{lem}\label{alpha diag}
Suppose that $H\subseteq [\mathbb{N}]^{<\infty}$ and for all $s\in H$, $s\cup\{\alpha_{s}\}\in \leftidx{^{*}}{H}$. Then for all $\vec{\alpha}$-trees $T$, if $st(T)\in H$ then there exists an $\vec{\alpha}$-tree $S\subseteq T$ with $st(S)=st(T)$ such that $S/st(S)\subseteq H$.
\end{lem}
\begin{proof}
Let $H\subseteq [\mathbb{N}]^{<\infty}$ such that for all $s\in H$, $s\cup\{\alpha_{s}\}\in \leftidx{^{*}}{H}$. Suppose that $T$ is an $\vec{\alpha}$-tree and $st(T)\in H$. We construct an $\vec{\alpha}$-tree $S$, level-by-level, recursively as follows  
$$\begin{cases}
L_{0} =\{st(T)\} \\
L_{n+1} = \{ s\cup \{m\} \in [\mathbb{N}]^{<\infty}: s\in L_{n}, \ m>\max(s) \ \& \ s\cup\{m\}\in H\cap T\}.
\end{cases}$$
 Since $T$ is an $\vec{\alpha}$-tree, for all $n\in \mathbb{N}$ and for all $s\in L_{n}$, $s\cup\{\alpha_{s}\}\in\leftidx{^{*}}{H}\cap\leftidx{^{*}}{T}=\leftidx{^{*}}{(H\cap T)}$. In particular, $s\cup\{\alpha_{s}\}\in \leftidx{^{*}}{L_{n+1}}$.  Let 
$$S = \{s \in [\mathbb{N}]^{<\infty} : s \sqsubseteq st(T)\} \cup \bigcup_{n=0}^{\infty} L_{n}.$$

It is clear that $S$ is a tree and $S\subseteq T$. The set $\{n\in\mathbb{N} : st(T)\cup \{n\}\in L_{1}\}$ is infinite since $st(T)\cup\{\alpha_{st(T)}\}\in \leftidx{^{*}}{L_{1}}$. Thus, $st(S)=st(T)$. If $s\in S/st(S)$ then there exists $n\in \mathbb{N}$ such that $s\in L_{n}$. So $s\cup \{\alpha_{s}\} \in \leftidx{^{*}}{L_{n+1}} \subseteq \leftidx{^{*}}{S}$. Hence, $S$ is an $\vec{\alpha}$-tree. Note that for all $n\in\mathbb{N}$, $L_{n}\subseteq H$. Thus
$$S/st(S) =\bigcup_{n=0}^{\infty} L_{n} \subseteq H.$$ \end{proof}
The proof of the next theorem does not appear in \cite{RamseySpaces}; however, the set $G$ does appear in the proof of Lemma 7.37 in \cite{RamseySpaces}.
\begin{thm}\label{alpha Ramsey theorem}
For all $\mathcal{X}\subseteq [\mathbb{N}]^{\infty}$ and for all $\vec{\alpha}$-trees $T$ there exists an $\vec{\alpha}$-tree $S\subseteq T$ with $st(S)=st(T)$ such that one of the following holds:
\begin{enumerate}
\item $[S]\subseteq \mathcal{X}$.
\item $[S]\cap \mathcal{X}=\emptyset$.
\item For all $\vec{\alpha}$-trees $S'$, if $S'\subseteq S$ then $[S']\not \subseteq \mathcal{X}$ and $[S']\cap\mathcal{X} \not = \emptyset$.
\end{enumerate}
\end{thm}
\begin{proof}
Suppose that $\mathcal{X}\subseteq[\mathbb{N}]^{\infty}$ and $T$ is an $\vec{\alpha}$-tree. Consider the following sets,
$$G= \{ s \in [\mathbb{N}]^{<\infty}: \exists \mbox{$\vec{\alpha}$-tree $S\subseteq T$ with stem $s$ such that $[S]\subseteq \mathcal{X}$}\}$$ 
$$F= \{ s \in [\mathbb{N}]^{<\infty}: \exists \mbox{$\vec{\alpha}$-tree $S\subseteq T$ with stem $s$ such that $[S]\subseteq [\mathbb{N}]^{\infty}\setminus\mathcal{X}$}\}$$
$$H= \{s\in [\mathbb{N}]^{<\infty} : \forall \mbox{$\vec{\alpha}$-tree $S\subseteq T$ with stem $s$, $[S]\not\subseteq \mathcal{X}$ and $[S]\cap \mathcal{X} \not = \emptyset$}\}.$$
Notice that $H= [\mathbb{N}]^{<\infty}\setminus(G \cup F)$.

\begin{claim}
If $s\cup \{\alpha_{s}\} \not\in \leftidx{^{*}}{H}$ then $s\not \in H$.
\end{claim}
\begin{proof}
Suppose that $s\in [\mathbb{N}]^{<\infty}$ and $s \cup \{\alpha_{s}\} \not \in \leftidx{^{*}}{H} = \leftidx{^{*}}{[\mathbb{N}]^{<\infty}\setminus(\leftidx{^{*}}{G} \cup \leftidx{^{*}}{F})}$. Hence, $s\cup \{\alpha_{s}\}\in \leftidx{^{*}}{G}$ or $s\cup \{\alpha_{s}\}\in \leftidx{^{*}}{F}$.

Consider the case when $s\cup \{\alpha_{s}\}\in \leftidx{^{*}}{G}$. For each $n \in \mathbb{N}$ such that $s\cup\{n\}\in G$, let $T_{n}$ be an $\vec{\alpha}$-tree with stem $s\cup\{n\}$ such that $[T_{n}]\subseteq \mathcal{X}$. Let $A= \{n\in\mathbb{N} : s\cup\{n\}\in G\}$ and note that $\alpha_{s} \in \leftidx{^{*}}{A}$. Let $S=\bigcup_{n\in A} T_{n}$. It is clear that $S$ is a tree with stem $s$, $\{s\cup\{n\} : n\in A\} =\bigcup_{n\in A}\{st(T_{n})\} \subseteq S$ and $[S] = \bigcup_{n\in A} [T_{n}] \subseteq \mathcal{X}$. If $t\in S$ then either $t=s$ or there exists $n\in A$ such that $t\in T_{n}/(s\cup\{n\})$. If $t=s$ then $t\cup\{\alpha_{t}\} = s \cup\{\alpha_{s}\} \in \leftidx{^{*}}{\{s\cup\{n\} : n\in A\}} =\leftidx{^{*}}{\bigcup_{n\in A}\{st(T_{n})\}} \subseteq \leftidx{^{*}}{S}.$ If there exists $n\in A$ such that $t\in T_{n}/(s\cup\{n\})$, then $t\cup\{\alpha_{t}\} \in \leftidx{^{*}}{T_{n}} \subseteq \leftidx{^{*}}{S}$. So $S$ is an $\vec{\alpha}$-tree with stem $s$ such that $[S]\subseteq \mathcal{X}$. Thus, $s\in G$. In particular, $s\not \in H$.

By an identical argument, if $s\cup \{\alpha_{s}\}\in \leftidx{^{*}}{F}$ then there is an $\vec{\alpha}$-tree $S$ with stem $s$ such that $[S]\subseteq ([\mathbb{N}]^{\infty} \setminus\mathcal{X}).$ In this case, we also have $s \not \in H$ as $s\in F$.\end{proof}

If $st(T)\in G$ then (1) holds. If $st(T)\in F$ then (2) holds. Otherwise $st(T)\in H$. By Lemma \ref{alpha diag} there is an $\vec{\alpha}$-tree $S\subseteq T$ such that $st(S)=st(T)$ and $S/st(S)\subseteq H$. If $S'\subseteq S$ is an $\vec{\alpha}$-tree then $st(S')\in S/st(S)\subseteq H$. Since $S'\subseteq T$, $[S']\not \subseteq \mathcal{X}$ and $[S']\cap \mathcal{X} \not=\emptyset$. So if (1) and (2) fail there is an $\vec{\alpha}$-tree showing that (3) holds.
\end{proof}
Although the previous theorem and its proof are quite simple they have many consequences including the following abstraction of Ramsey's Theorem to $\vec{\alpha}$-trees.
\begin{cor}[$\vec{\alpha}$-Ramsey Theorem]\label{alpha-tree Ramsey theorem}
Suppose that $n\in \mathbb{N}$. For all $A\subseteq [\mathbb{N}]^{n}$ and for all $\vec{\alpha}$-trees $T$ there exists an $\vec{\alpha}$-tree $S\subseteq T$ with $st(S)=st(T)$ such that either $S(n)\subseteq A$ or $S(n)\cap A =\emptyset$.
\end{cor}
\begin{proof}
Let $n\in \mathbb{N}$, $A\subseteq[\mathbb{N}]^{n}$ and $T$ be an $\vec{\alpha}$-tree. For each $Y\subseteq \mathbb{N}$ with $|Y|\ge n$, let $r_{n}(Y)$ denote the $n$-smallest elements of $Y$ \emph{i.e.} the first $n$ natural numbers in the increasing enumeration of $Y$. For example, $r_{n}(\mathbb{N}) = \{0,1,2,\dots,n-1\}$. Let $\mathcal{X}=\{ Y\in[\mathbb{N}]^{\infty} : r_{n}(Y)\in A\}$. Notice that $\mathcal{X}$ can not satisfy conclusion (3) in the statement of Theorem \ref{alpha Ramsey theorem} because any $\vec{\alpha}$-tree $S$ with $|st(S)|\ge n$ will either have $[S]\subseteq \mathcal{X}$ or $[S]\cap \mathcal{X} = \emptyset$ depending on whether $r_{n}(st(S))\in A$ or $r_{n}(st(S))\not\in A$, respectively. So Theorem \ref{alpha Ramsey theorem} implies that there is an $\vec{\alpha}$-tree $S\subseteq T$ with $st(S)=st(T)$ such that either $[S]\subseteq\mathcal{X}$ or $[S]\cap \mathcal{X}=\emptyset.$ Thus, either $S(n)\subseteq A$ or $S(n)\cap A=\emptyset$ depending on whether $[S]\subseteq\mathcal{X}$ or $[S]\cap \mathcal{X}=\emptyset,$ respectively.
\end{proof}

\subsection{Local Ramsey theory} We use Theorem \ref{alpha Ramsey theorem} to construct an $\vec{\alpha}$-Ramsey theory that in its development runs parallel to local Ramsey theory. In this section we introduce the basics of local Ramsey theory and its cornerstone result, the Ellentuck Theorem. 

For $s\in[\mathbb{N}]^{<\infty}$ and $X\in[\mathbb{N}]^{\infty}$, we use the following notation:
$$[s]= \{ Y \in \mathbb{N}]^{\infty} : s\sqsubseteq Y\},$$
$$[s,X] =\{Y \in [\mathbb{N}]^{\infty} : s \sqsubseteq Y \subseteq X\}.$$
The \emph{metric topology on $[\mathbb{N}]^{\infty}$} is the compact metrizable zero-dimensional topology generated by sets of the form $[s]$. The \emph{Ellentuck Space} is the zero-dimensional topological space on $[\mathbb{N}]^{\infty}$ with the countable chain condition generated by the sets of the form $[s,X]$. Every metrically open set is also open with respect to the Ellentuck space since $[s]=[s,\mathbb{N}]$.

\begin{defn}[Baire, \cite{BaireDiss}]
Let $(X,\tau)$ be a topological space and $\mathcal{B}$ be a basis for the topology. $Y\subseteq X$ is \emph{nowhere dense} if for all nonempty $U\in \mathcal{B}$ there exists a nonempty $V\subseteq U$ such that $V\in\mathcal{B}$ and $V\cap Y = \emptyset$. $Y \subseteq X$ is \emph{meager} if it is the countable union of nowhere dense sets. $\mathcal{X}$ \emph{has the Baire property} if $X=U \Delta M$ where $U$ is open and $M$ is meager.
\end{defn}

\begin{defn}
Suppose that $X$ is a set. A \emph{$\sigma$-ideal on X} is a collection of sets $\mathcal{I}\subseteq \wp(X)$ such that $\emptyset \in I$, $I$ is closed under countable unions and if $Y\subseteq Z\in I$ then $Y\in I$. A \emph{$\sigma$-algebra on X} is a collection of sets $\mathcal{A}\subseteq \wp(X)$ such that $\emptyset, X\in \mathcal{A}$, $\mathcal{A}$ is closed under complements, countable unions and countable intersections.
\end{defn}

\begin{thm}[Baire, \cite{BaireDiss}] For all topological spaces the collection of sets with the Baire property with respect to the space is a $\sigma$-algebra. Moreover, the meager sets with respect to the space form a $\sigma$-ideal contained in the $\sigma$-algebra of sets with the Baire property.
\end{thm}

\begin{defn}
Suppose that $\mathcal{C}\subseteq [\mathbb{N}]^{\infty}$. $\mathcal{X}\subseteq [\mathbb{N}]^{\infty}$ is \emph{$\mathcal{C}$-Ramsey} if for all $[s,X]\not=\emptyset$ with $X\in\mathcal{C}$ there exists $Y\in[s,X]\cap \mathcal{C}$ such that either $[s,Y]\subseteq \mathcal{X}$ or $[s,Y]\cap \mathcal{X}=\emptyset.$ $\mathcal{X}\subseteq [\mathbb{N}]^{\infty}$ is \emph{$\mathcal{C}$-Ramsey null} if for all $[s,X]\not=\emptyset$ with $X\in\mathcal{C}$ there exists $Y\in[s,X]\cap\mathcal{C}$ such that $[s,Y]\cap \mathcal{X}=\emptyset.$
\end{defn}

\begin{defn} Suppose that $\mathcal{C}\subseteq [\mathbb{N}]^{\infty}$. We say that $([\mathbb{N}]^{\infty}, \mathcal{C}, \subseteq)$ is a \emph{topological Ramsey space} if the following conditions hold:
\begin{enumerate}
\item $\{[s,X] :X\in\mathcal{C}\}$ is a neighborhood base for a topology on $[\mathbb{N}]^{\infty}$.
\item The collection of $\mathcal{C}$-Ramsey sets coincides with the $\sigma$-algebra of sets with the Baire property with respect to the topology generated by $\{[s,X] :X\in\mathcal{C}\}$.
\item The collection of $\mathcal{C}$-Ramsey null sets coincides with the $\sigma$-ideal of meager sets with respect to the topology generated by $\{[s,X] :X\in\mathcal{C}\}$.
\end{enumerate}
\end{defn}

Local Ramsey theory is concerned with characterizing the conditions on $\mathcal{C}$ which guarantee that $([\mathbb{N}]^{\infty}, \mathcal{C}, \subseteq)$ forms a topological Ramsey space. We will use the $\vec{\alpha}$-Ellentuck Theorem to provide such a characterization when $\mathcal{C}$ is taken to be an ultrafilter on $\mathbb{N}$. The $\vec{\alpha}$-Ellentuck Theorem, which we prove later, is a generalization of the Ellentuck Theorem.
\begin{thm}[The Ellentuck Theorem, \cite{Ellentuck}]
$([\mathbb{N}]^{\infty}, [\mathbb{N}]^{\infty}, \subseteq)$ is a topological Ramsey space.
\end{thm}

The statement of the $\vec{\alpha}$-Ellentuck Theorem is very similar; instead of taking $[s,X]$ to be the basic open sets we take them to be $\vec{\alpha}$-trees. Under certain saturation assumptions on the Alpha-Theory, $\vec{\alpha}$-trees coincide with $\vec{\mathcal{U}}$-trees as introduced by Blass in \cite{BlassU-trees}. In this way, the proof of the $\vec{\alpha}$-Ellentuck Theorem can be seen as a proof, within the Alpha-Theory, of the Ultra-Ellentuck Theorem developed by Todorcevic in \cite{RamseySpaces}. In the final section of this paper we extend these results, to the abstract setting, in the same way that the Abstract Ellentuck Theorem extends the Ellentuck Theorem. 

\subsection{The $\vec{\alpha}$-Ellentuck Theorem}
For some subsets of $[\mathbb{N}]^{\infty}$ conclusion (1) and (3) of Theorem \ref{alpha Ramsey theorem} are impossible. For example, every one-element subset of $[\mathbb{N}]^{\infty}$ has the property that conclusion (1) and (3) are impossible. On the other hand, for some subsets (1) and (2) are possible but (3) is impossible. For example, for all $\vec{\alpha}$-trees $T$, (1) and (2) are possible for $[T]$ but (3) is impossible for $[T]$.

\begin{defn}
$\mathcal{X}\subseteq [\mathbb{N}]^{\infty}$ is said to be \emph{$\vec{\alpha}$-Ramsey} if for all $\vec{\alpha}$-trees $T$ there exists an $\vec{\alpha}$-tree $S\subseteq T$ with $st(S)=st(T)$ such that either $[S]\subseteq \mathcal{X}$ or $[S]\cap \mathcal{X}=\emptyset$. $\mathcal{X}$ is said to be \emph{$\vec{\alpha}$-Ramsey null} if for all $\vec{\alpha}$-trees $T$ there exists an $\vec{\alpha}$-tree $S\subseteq T$ with $st(S)=st(T)$ such that $[S]\cap \mathcal{X}=\emptyset$.
\end{defn}

\begin{cor}\label{sigma ideal}
The collection of $\vec{\alpha}$-Ramsey null sets is a $\sigma$-ideal.
\end{cor}
\begin{proof}
It is clear that $\emptyset$ is $\vec{\alpha}$-Ramsey null. It should also be clear from the definition that if $\mathcal{X}$ is $\vec{\alpha}$-Ramsey null and $\mathcal{Y}\subseteq \mathcal{X}$ then $\mathcal{Y}$ is also $\vec{\alpha}$-Ramsey null. So it is enough to show that the countable union of $\vec{\alpha}$-Ramsey null sets is $\vec{\alpha}$-Ramsey null. To this end, let $\left<\mathcal{X}_{i}:i\in\mathbb{N}\right>$ be a sequence of $\vec{\alpha}$-Ramsey null sets. 

Let $T$ be an $\vec{\alpha}$-tree. Let $S_{st(T)}\subseteq T$ be an $\vec{\alpha}$-tree such that $[S_{st(T)}]\cap \mathcal{X}_{0} =\emptyset$, such a tree exists as $\mathcal{X}_{0}$ is $\vec{\alpha}$-Ramsey null. Suppose that $S_{t}$ has been defined, is an $\vec{\alpha}$-tree, $st(S_{t})=t$ and $[S_{t}]\cap\mathcal{X}_{|t|-|st(T)|} =\emptyset$; note that $|t|-|st(T)|$ is a non-negative integer. For each $s\in S_{t}/t$ such that $|s|=|t|+1$, since $\mathcal{X}_{|s|-|st(T)|}$ is $\vec{\alpha}$-Ramsey null, there exists an $\vec{\alpha}$-tree $S_{s}\subseteq S_{t}$ such that $st(S_{s})=s$ and $[S_{s}]\cap \mathcal{X}_{|s|-|st(T)|}=\emptyset$. Then let
$$\begin{cases}
L_{0} = \{ st(T)\},\\
L_{n+1} =\{t\cup\{m\}\in [\mathbb{N}]^{<\infty} : t \in L_{n}, \ m>\max(t) \ \& \ t\cup\{m\}\in S_{t}\}. 
\end{cases}$$
Let 
$$ S = \{ t\in[\mathbb{N}]^{<\infty} : t\sqsubseteq st(T) \} \cup \bigcup_{n=0}^{\infty} L_{n}.$$
It is clear that $S$ is a tree with $st(S)=st(T)$. If $s\in S/st(S)$ then there exists $n\in\mathbb{N}$ such that $s\in L_{n}$. Since $S_{s}$ is an $\vec{\alpha}$-tree $s\cup\{\alpha_{s}\}\in \leftidx{^{*}}{S_{s}}$. Hence, $s\cup\{\alpha_{s}\} \in \leftidx{^{*}}{L_{n+1}}\subseteq \leftidx{^{*}}{S}$. So $S$ is an $\vec{\alpha}$-tree. 

If $X\in[S]$ and $n\in \mathbb{N}$ then there exists $t\sqsubseteq X$ such that $|t|=|st(T)|+n$. If $t \sqsubset t' \sqsubseteq X$ then $t'\in S_{t'\setminus{\max(t')}} \subseteq S_{t}$. Since $S_{t}$ is an $\vec{\alpha}$-tree with $st(S_{t})=t$, $X\in [S_{t}]$. Note that $[S_{t}]\cap \mathcal{X}_{n} =[S_{t}]\cap \mathcal{X}_{|t|-|st(T)|} = \emptyset$. Hence, $X\not \in \mathcal{X}_{n}$. Since $X$ was an arbitrary element of $[S]$ and $n$ an arbitrary element of $\mathbb{N}$,
$$[S]\cap \bigcup_{n=0}^{\infty} \mathcal{X}_{n} = \emptyset.$$
Thus, $\bigcup_{n=0}^{\infty} \mathcal{X}_{n}$ is $\vec{\alpha}$-Ramsey null.
\end{proof}
\begin{cor}
$[\mathbb{N}]^{\infty}$ is uncountable. 
\end{cor}
\begin{proof}
Let $X\in [\mathbb{N}]^{\infty}$. For all $\vec{\alpha}$-trees $T$, $[T]$ is infinite. Thus for all $\vec{\alpha}$-trees $T$, $[T]\not\subseteq \{X\}$. Therefore conclusion (1) in Theorem \ref{alpha Ramsey theorem} is not possible for $\{X\}$. Conclusion (3) is also not possible as it implies that there is an $\vec{\alpha}$-tree $S$ such that for all $\vec{\alpha}$-trees $S'\subseteq S$, $[S']\cap\{X\}\not=\emptyset$, a contradiction. So, Theorem \ref{alpha Ramsey theorem} implies that $\{X\}$ is $\vec{\alpha}$-Ramsey null.

Toward a contradiction suppose that $[\mathbb{N}]^{\infty}$ is countable. By Corollary \ref{sigma ideal}, $[\mathbb{N}]^{\infty}$ would be the countable union of $\vec{\alpha}$-Ramsey null sets and hence Ramsey null itself. This is a contradiction since for all $\vec{\alpha}$-trees $S$, $[S]\subseteq [\mathbb{N}]^{\infty}$. 
\end{proof}
The next lemma, among other things, shows that for all $\vec{\alpha}$-trees $T$, $[T]$ is $\vec{\alpha}$-Ramsey. The lemma in the context of $\vec{\mathcal{U}}$-trees appears as Lemma 7.31 without proof in \cite{RamseySpaces}; however, the statement has a typo. Lemma 7.31 in \cite{RamseySpaces} states that the intersection gives an $\vec{\mathcal{U}}$-tree if and only if the stems are $\sqsubseteq$-comparable. It is possible to give examples of $\vec{\mathcal{U}}$-tress with $\sqsubseteq$-comparable stems whose intersection is not an $\vec{\mathcal{U}}$-tree.

\begin{lem}\label{intersection alpha trees}
If $T$ and $S$ are $\vec{\alpha}$-trees then $S\cap T$ is an $\vec{\alpha}$-tree if and only if either $st(T)\sqsubseteq st(S)\in T$ or $st(S)\sqsubseteq st(T)\in S$.
\end{lem}
\begin{proof} Let $S$ and $T$ be $\vec{\alpha}$-trees. If $st(T)$ and $st(S)$ are not $\sqsubseteq$-comparable then $S\cap T =\{t\in[\mathbb{N}]^{<\infty} : t\sqsubseteq st(S) \ \& \ t\sqsubseteq st(T)\}$ and can not be an $\vec{\alpha}$-tree as it is a finite set. By contrapostive, if $S\cap T$ is an $\vec{\alpha}$-tree then either $st(S)\sqsubseteq st(T)$ or $st(T)\sqsubseteq st(S)$.  If $st(S)\sqsubseteq st(T)$ and $st(T)\not\in S$ then $S\cap T\subseteq \{t\in[\mathbb{N}]^{<\infty} : t\sqsubseteq st(T)\}$ can not be an $\vec{\alpha}$-tree as it is a finite set. Hence, if $S\cap T$ is an $\vec{\alpha}$-tree and $st(S)\sqsubseteq st(T)$ then $st(T) \in S$. Likewise, if $S\cap T$ is an $\vec{\alpha}$-tree and $st(T)\sqsubseteq st(S)$ then $st(S)\in T$. Altogether we have shown that if $S\cap T$ is an $\vec{\alpha}$-tree then either $st(S)\sqsubseteq st(T)\in S$ or $st(T)\sqsubseteq st(s)\in T$.

 Let $S$ and $T$ be $\vec{\alpha}$-trees and suppose $st(S)\sqsubseteq st(T)\in S$. Then $ (S\cap T)/st(S \cap T) =(S\cap T)/st(T) = S/st(T) \cap T/st(T) = S/st(S) \cap T/st(T)\not=\emptyset$.  Since $S$ and $T$ are both $\vec{\alpha}$-trees, for each $s\in (S\cap T)/st(S\cap T)$, $s\cup\{\alpha_{s}\} \in \leftidx{^{*}}{S} \cap \leftidx{^{*}}{T}=\leftidx{^{*}}{(S\cap T)}$. That is, $S\cap T$ is an $\vec{\alpha}$-tree. By symmetry, if either $st(T)\sqsubseteq st(S)\in T$ or $st(S)\sqsubseteq st(T)\in S$ then $S\cap T$ is an $\vec{\alpha}$-tree.
\end{proof}

The $\vec{\alpha}$-Ramsey and $\vec{\alpha}$-Ramsey null sets can be completely characterized in terms of topological notions with respect to a space on $[\mathbb{N}]^{\infty}$ generated from the $\vec{\alpha}$-trees. 

\begin{lem}
The collection $\{[T]: \mbox{$T$ is an $\vec{\alpha}$-tree}\}$ is a basis for a topology on $[\mathbb{N}]^{\infty}$.
\end{lem}
\begin{proof}
It is enough to show that, if $S$ and $T$ are $\vec{\alpha}$-trees and $[S]\cap[T]\not=\emptyset$ then $S\cap T$ is an $\vec{\alpha}$-tree.
Suppose $S$ and $T$ are $\vec{\alpha}$-trees and $X\in [S]\cap[T]$. Then either $st(T)\sqsubseteq st(S) \sqsubseteq X$ or $st(S)\sqsubseteq st(T) \sqsubseteq X$. If $st(S)\sqsubseteq X$ then $st(S) \in T$ since $X\in[T]$. Likewise, if $st(T)\sqsubseteq X$ then $st(T) \in S$. So, either $st(S)\sqsubseteq st(T)\in S$ or $st(T)\sqsubseteq st(S)\in T$. By Lemma \ref{intersection alpha trees}, $S\cap T$ is an $\vec{\alpha}$-tree.
\end{proof}

\begin{defn}
The topology on $[\mathbb{N}]^{\infty}$ generated by $\{[T] : \mbox{T is an $\vec{\alpha}$-tree}\}$ is called \emph{the $\vec{\alpha}$-Ellentuck topology}. We say that a subset of $[\mathbb{N}]^{\infty}$ is \emph{$\vec{\alpha}$-open} if it is open with respect the $\vec{\alpha}$-Ellentuck topology.
\end{defn}

\begin{rem}
We leave it to the interested reader to show that the \emph{$\vec{\alpha}$-Ellentuck space} is a zero-dimensional Baire space on $[\mathbb{N}]^{\infty}$ with the countable chain condition. 
\end{rem}

\begin{cor}\label{open is ramsey}
Every $\vec{\alpha}$-open set is $\vec{\alpha}$-Ramsey.
\end{cor}
\begin{proof} We show that, if $\mathcal{X}\subseteq[\mathbb{N}]^{\infty}$ is not $\vec{\alpha}$-Ramsey then there exists $X\in \mathcal{X}$ such that for all $\vec{\alpha}$-trees $S$, if $X\in [S]$ then  $[S]\not\subseteq \mathcal{X}$. In other words, if $\mathcal{X}$ is not $\vec{\alpha}$-Ramsey then it contains a point not in its interior with respect to the $\vec{\alpha}$-Ellentuck topology. The result follows by taking the contrapositve of this statement.

Suppose that $\mathcal{X}$ is not $\vec{\alpha}$-Ramsey. Then there exists an $\vec{\alpha}$-tree $T$ such that for each $\vec{\alpha}$-tree $S\subseteq T$ with $st(S)=st(T)$, $[S]\not\subseteq \mathcal{X}$ and $[S]\cap \mathcal{X}\not = \emptyset$. By Theorem \ref{alpha Ramsey theorem} there is an $\vec{\alpha}$-tree $S\subseteq T$ with $st(S)=st(T)$ such that for all $\vec{\alpha}$-trees $S'\subseteq S$, $[S']\not \subseteq \mathcal{X}$ and $[S']\cap \mathcal{X}\not=\emptyset$. 

Since $S\subseteq S$, $[S]\cap \mathcal{X} \not =\emptyset$. Let $X$ be any element of $[S]\cap \mathcal{X}$. If $S'$ is an $\vec{\alpha}$-tree and $X\in[S']$ then $X\in[S]\cap [S']$. So either $st(S')\sqsubseteq st(S) \sqsubseteq X$ or $st(S)\sqsubseteq st(S') \sqsubseteq X$. If $st(S)\sqsubseteq X$ then $st(S) \in S'$ since $X\in[S']$. Likewise, if $st(S')\sqsubseteq X$ then $st(S') \in S$. By Lemma \ref{intersection alpha trees}, $S'\cap S$ is an $\vec{\alpha}$-tree. Since $S'\cap S\subseteq S$, the previous paragraph shows that $[S'\cap S]\not \subseteq \mathcal{X}$ and $[S'\cap S] \cap \mathcal{X} \not = \emptyset$. In particular, $[S']\not \subseteq \mathcal{X}$ as $[S'\cap S]\subseteq [S]$.
\end{proof}

\begin{defn}
$\mathcal{X}\subseteq [\mathbb{N}]^{\infty}$ \emph{is $\vec{\alpha}$-nowhere dense/ is $\vec{\alpha}$-meager/  has the $\vec{\alpha}$-Baire property} if it is nowhere dense/ is meager/ has the Baire property with respect to the $\vec{\alpha}$-Ellentuck topology. 

We say that $([\mathbb{N}]^{\infty},\vec{\alpha},\subseteq)$ is a \emph{$\vec{\alpha}$-Ramsey space} if the collection of $\vec{\alpha}$-Ramsey sets coincides with the $\sigma$-algebra of sets with the $\vec{\alpha}$-Baire property and the collection of $\vec{\alpha}$-Ramsey null sets coincides with the $\sigma$-ideal of $\vec{\alpha}$-meager sets.
\end{defn}

The next theorem is equivalent to the ultra-Ellentuck Theorem of Todorcevic under the assumption, which we explore later, of the $\mathfrak{c}^{+}$-enlarging property.

\begin{thm} (The $\vec{\alpha}$-Ellentuck Theorem)
$([\mathbb{N}]^{\infty},\vec{\alpha},\subseteq)$ is a $\vec{\alpha}$-Ramsey space.
\end{thm}
\begin{proof} First note that it is clear from the definitions that every $\vec{\alpha}$-Ramsey null set is $\vec{\alpha}$-nowhere dense. Let $\mathcal{X}\subseteq [\mathbb{N}]^{\infty}$ be $\vec{\alpha}$-nowhere dense and $T$ be an $\vec{\alpha}$-tree. Note that conclusion (1) in Theorem \ref{alpha Ramsey theorem} is not possible for $\mathcal{X}$ because otherwise it would not be $\vec{\alpha}$-nowhere dense. Similarly, conclusion (3) in Theorem \ref{alpha Ramsey theorem} is not possible for $\mathcal{X}$ because otherwise it would not be $\vec{\alpha}$-nowhere dense. So, by Theorem \ref{alpha Ramsey theorem}, there exists an $\vec{\alpha}$-tree $S\subseteq T$ with $st(S)=s(T)$ such that $[S]\cap\mathcal{X}=\emptyset$. Hence, $\mathcal{X}$ is $\vec{\alpha}$-Ramsey null. 

If $\mathcal{X}$ is $\vec{\alpha}$-meager then $\mathcal{X}$ is the countable union of $\vec{\alpha}$-nowhere dense sets. The previous paragraph and Corollary \ref{sigma ideal} imply that $\mathcal{X}$ is $\vec{\alpha}$-Ramsey null. On the other hand, if $\mathcal{X}$ is $\vec{\alpha}$-Ramsey null then it is also $\vec{\alpha}$-meager since the last paragraph implies that it is $\vec{\alpha}$-nowhere dense. Therefore, the collection of $\vec{\alpha}$-meager sets coincides with the $\sigma$-ideal of $\vec{\alpha}$-Ramsey null sets. 

Suppose that $\mathcal{X}$ has the $\vec{\alpha}$-Baire property. Then there is an $\vec{\alpha}$-open set $\mathcal{O}$ and an $\vec{\alpha}$-meager set $\mathcal{M}$ such that $\mathcal{X} = \mathcal{O} \Delta \mathcal{M}$. Corollary \ref{open is ramsey} and the previous paragraph imply that $\mathcal{O}$ and $\mathcal{M}$ are $\vec{\alpha}$-Ramsey and $\vec{\alpha}$-Ramsey null, respectively. If $T$ is an $\vec{\alpha}$-tree then there exists $\vec{\alpha}$-trees $S'\subseteq S\subseteq T$ with $st(S')=st(S)=st(T)$ such that $[S]\cap \mathcal{M} =\emptyset$ and either $[S']\subseteq \mathcal{O}$ or $[S']\cap \mathcal{O}=\emptyset$. So either, $[S'] \cap (\mathcal{O}\Delta \mathcal{M})=\emptyset$ or $[S']\subseteq \mathcal{O}\Delta \mathcal{M}$. Hence, $\mathcal{X}$ is $\vec{\alpha}$-Ramsey.

Let $\mathcal{X}$ be an $\vec{\alpha}$-Ramsey set. Let $\mathcal{O}$ be the interior of $\mathcal{X}$ with respect to the $\vec{\alpha}$-Ellentuck topology.  $\mathcal{O}$ is $\vec{\alpha}$-Ramsey by Corollary \ref{open is ramsey}. So, $\mathcal{X} \setminus \mathcal{O}$ is $\vec{\alpha}$-Ramsey by Corollary \ref{sigma ideal}. So, for all $\vec{\alpha}$-trees $T$ there exists an $\vec{\alpha}$-tree $S\subseteq T$ such that either $[S]\subseteq \mathcal{X}\setminus \mathcal{O}$ or $[S]\cap (\mathcal{X}\setminus\mathcal{O})=\emptyset$. $[S]\subseteq \mathcal{X}\setminus \mathcal{O}$ is not possible because it would mean that $[S]\subseteq\mathcal{O}$ as it is an $\vec{\alpha}$-open set contained in $\mathcal{X}$. Thus, $[S]\cap (\mathcal{X}\setminus\mathcal{O})=\emptyset$. So $\mathcal{X}\setminus\mathcal{O}$ is $\vec{\alpha}$-Ramsey null. By the second paragraph of this proof, $\mathcal{X}\setminus \mathcal{O}$ is $\vec{\alpha}$-meager. $\mathcal{X}$ has the $\vec{\alpha}$-Baire property since $\mathcal{X} = \mathcal{O} \Delta (\mathcal{X}\setminus \mathcal{O})$. 

The previous two paragraphs show that the collection of sets with the $\vec{\alpha}$-Baire property coincides with the $\sigma$-algebra of $\vec{\alpha}$-Ramsey sets. 
\end{proof}

\subsection{$\beta$-Ramsey theory}
If there exists a hypernatural number $\beta$ such that for all $s\in [\mathbb{N}]^{<\infty}$, $\alpha_{s}=\beta$ then we suppress the arrow and denote $\vec{\alpha}=\left< \alpha_{s} :s\in[\mathbb{N}]^{<\infty}\right>$ by $\beta$. For example, $T$ is a $\beta$-tree if and only if $T$ is a tree with stem $st(T)$ such that $T/st(T)\not=\emptyset$ and for all $s\in T/st(T)$, $s\cup\{\beta\} \in  \leftidx{^{*}}{T}.$

\begin{lem}\label{diag lem}
Suppose that $\beta$ is a nonstandard hypernatural number and $T$ is a $\beta$-tree. Then for all $s\in T/st(T)$ there exists $X\in[s,\mathbb{N}]$ such that $[s,X] \subseteq [T].$
\end{lem}
\begin{proof} Let $T$ be a $\beta$-tree. We first construct a sequence of infinite subsets of $\mathbb{N}$ and then diagonalize them to obtain $X$. Since $s\cup\{\beta\}\in \leftidx{^{*}}{T}$ there is an infinite set $X$ such that $\beta\in\leftidx{^{*}}{X}$ and  for all $x\in X$, $s\cup\{x\}\in T$. Let $X_{0}$ be any such infinite set. Now suppose that $(X_{i})_{i<n}$ is a sequence of infinite sets such that $X_{0}\supseteq X_{1} \supseteq \dots \supseteq X_{n-1}$, $\beta\in \leftidx{^{*}}{X_{n}}$ and for all finite sets $\{x_{0}, x_{1}, \dots, x_{n-1}\}$ listed in increasing order, if for all $i<n$, $x_{i} \in X_{i}$ then $s\cup\{x_{0}, x_{1}, \dots, x_{n-1}\} \in T$. Because $T$ is a $\beta$-tree, for all $i<n$ and for all finite sets $\{x_{0}, x_{1}, \dots, x_{n-1}\}$ listed in increasing order, if for all $i<n$, $x_{i} \in X_{i}$ then $s\cup\{x_{0}, x_{1}, \dots, x_{n-1}\}\cup\{\beta\} \in \leftidx{^{*}}{T}$. Therefore, there is an infinite set $Y\subseteq X_{n-1}$ such that $\beta\in \leftidx{^{*}}{Y}$ and for all finite sets $\{x_{0}, x_{1}, \dots, x_{n-1}\}$ listed in increasing order, if for all $i<n$, $x_{i} \in X_{i}$ and $x_{n}\in \{x\in Y: x>x_{n-1}\}$ then $s\cup\{x_{0}, x_{1}, \dots, x_{n-1}\}\cup \{x_{n}\} \in T$. Let $X_{n}=\{x\in Y: x>x_{n-1}\}$. This completes the construction of $\left<X_{i}:i\in \mathbb{N}\right>$. 

Now we obtain $X$ by diagonalizing this sequence, let 
$$\begin{cases}
x_{0} = \min \{x\in X_{0} : x >\max(s)\}  \\
x_{i+1} = \min \{x\in X_{i} : x >x_{i}\} 
\end{cases}$$
and $X= s\cup\{x_{0}, x_{1}, x_{2}, \dots\}$. Suppose that $Y$ is an infinite set such that $s\sqsubseteq Y \subseteq X$. If $t$ is an initial segment of $Y$ and $s\sqsubset t$ then there exists $n\in \mathbb{N}$ and a finite set $\{y_{0},y_{1}, \cdots, y_{n}\}\subseteq X$, listed in increasing order, such that $t=s\cup \{y_{0},y_{1}, \cdots, y_{n}\}$ and for all $i<n+1$, $y_{i} \in X_{i}$. Thus $t=s\cup \{y_{0},y_{1}, \cdots, y_{n}\}\in T$. If $t$ is an initial segments of $Y$ and $t\sqsubseteq s$ then $s\in T$ because $T$ is a tree. So, any initial segment of $Y$ is in $T$. In other words, $Y\in [T]$. Therefore,
$[s,X] \subseteq [T].$
\end{proof}

\begin{lem} \label{basic open to a tree}
If $[s,X]\not=\emptyset$ then for all nonstandard hypernatural numbers $\beta\in \leftidx{^{*}}{X}$, there exists a $\beta$-tree $T$ with stem $s$ such that $[T]= [s,X]$. 
\end{lem}
\begin{proof} Let $\beta$ be a nonstandard hypernatural number such that $\beta\in \leftidx{^{*}}{X}$. Construct a $\beta$-tree, level-by-level, as follows
$$\begin{cases}
L_{0} = \{s\},\\
L_{i+1} = \{ s\cup\{n\} \in [\mathbb{N}]^{<\infty}: s\in L_{i}, n>\max(s) \ \& \ n\in X\}.
\end{cases}$$
Let $T$ denote the tree $\{t\in[\mathbb{N}]^{<\infty}: t\sqsubseteq t\}\cup\bigcup_{i=0}^{\infty} L_{i}$. $T$ is an $\beta$-tree with $st(T)=\emptyset$ since $\beta\in \leftidx{^{*}}{X}$. It is clear that $[T]=[s,X]$.
\end{proof}

The next result shows that Ramsey's Theorem follows as a direct consequence of Theorem \ref{alpha Ramsey theorem} and the previous two lemmas.
\begin{cor}[Ramsey's Theorem, \cite{RamseyThm}]
Let $n\in\mathbb{N}$ and $A\subseteq [\mathbb{N}]^{n}$. For all $X\in[\mathbb{N}]^{\infty}$ there exists an $Y\in[X]^{\infty}$ such that either $[Y]^{n}\subseteq A$ or $[Y]^{n}\cap A = \emptyset$.
\end{cor}
\begin{proof}
Let $n\in \mathbb{N}$, $A\subseteq[\mathbb{N}]^{n}$ and $X\in[\mathbb{N}]^{\infty}$. By Lemma \ref{basic open to a tree}, applied to $[\emptyset,X]$, there exists $\beta\in \leftidx{^{*}}{X}\setminus X$ and a $\beta$-tree $T$ with stem $\emptyset$ such that $[T]\subseteq [\emptyset,X]$. Corollary \ref{alpha-tree Ramsey theorem} implies that there is a $\beta$-tree $S\subseteq T$ with $st(S)=\emptyset$ such that either $S(n)\subseteq A$ or $S(n)\cap A=\emptyset.$ By Lemma \ref{diag lem} there exist $Y\in[\mathbb{N}]^{\infty}$ such that $[\emptyset, Y] \subseteq [S]$. If $S(n)\subseteq A$ then $[Y]^{n}\subseteq A$ since for all $s\in[Y]^{n}$ there exists $Z\in[Y]^{\infty}$ such that $r_{n}(Z)=s$. If $S(n)\cap A=\emptyset$ then $[Y]^{n}\cap A=\emptyset$ since for all $s\in[Y]^{n}$ there exists $Z\in[Y]^{\infty}$ such that $r_{n}(Z)=s$. Note that $[\emptyset, Y]\subseteq [S]\subseteq [T] \subseteq [\emptyset, X]$. Hence, $Y\in [X]^{\infty}$.
\end{proof}

\begin{thm}\label{beta-SCIP} The following are equivalent:
\begin{enumerate}
\item If $T$ is a $\beta$-tree then for all $s\in T/st(T)$ there exists $X\in[s,\mathbb{N}]$ such that $\beta\in\leftidx{^{*}}{X}$ and $[s,X] \subseteq [T].$
\item For each sequence of sets $X_{0}\supseteq X_{1}\supseteq X_{2} \supseteq \cdots$ such that $\beta\in\bigcap_{i=0}^{\infty} \leftidx{^{*}}{X_{i}}$, there exists $X=\{x_{0},x_{1},\dots\}$ enumerated in increasing order such that $\beta\in \leftidx{^{*}}{X}$ and for all $n\in\mathbb{N}$, $ x_{n+1}\in X_{n}.$
\end{enumerate}
\end{thm}
\begin{proof}
(2)$\implies$(1). In the proof of Lemma \ref{diag lem} instead of constructing a diagonalization use (2).

(1)$\implies$(2). Suppose that $X_{0}\supseteq X_{1}\supseteq X_{2} \supseteq \cdots$ and $\beta\in\bigcap_{i=0}^{\infty} \leftidx{^{*}}{X_{i}}$. Consider the following set
$$\mathcal{X} = \{Y\in[\mathbb{N}]^{\infty}: \{y_{0},y_{1}\}\sqsubseteq Y,\ y_{0}<y_{1} \ \& \ y_{1}\in X_{y_{0}}\}.$$
If $Y\in\mathcal{X}$ and $\{y_{0},y_{1}\}$ are the two smallest elements of $Y$ listed in increasing order, then $y_{1}\in X_{y_{0}}$. So $Y\in[\{y_{0},y_{1}\}, \mathbb{N}]\subseteq \mathcal{X}$. In particular, $\mathcal{X}$ is  $\beta$-open by Lemma \ref{basic open to a tree}. By the $\beta$-Ellentuck theorem there exists a $\beta$-tree $T$ with $st(T)=\emptyset$ such that either $[T]\subseteq \mathcal{X}$ or $[T]\cap \mathcal{X}=\emptyset$. So by $(1)$ there exists $X=\{x_{0},x_{1},x_{2},\dots\}$ listed in increasing order such that $\alpha\in \leftidx{^{*}}{X}$ and either $[\emptyset,X]\subseteq \mathcal{X}$ or $[\emptyset,X] \cap \mathcal{X}=\emptyset$. 

If $[\emptyset, X]\cap \mathcal{X}=\emptyset$ then for all $i\in\mathbb{N}$, $x_{0}\not \in X_{ x_{i+1}}$. Thus $\leftidx{^{*}}{((X\setminus\{x_{0}\}) \cap X_{x_{0}}\})}=\emptyset$. But this is a contradiction since $\beta\in  \leftidx{^{*}}{((X\setminus\{x_{0}\}) \cap X_{x_{0}}\})}$. So $[\emptyset, X]\subseteq \mathcal{X}$. In particular, for all $n\in\mathbb{N}$, $x_{n+1}\in X_{x_{n}}\subseteq X_{n}$ and $\beta\in \leftidx{^{*}}{X}$.
\end{proof}

\subsection{Some applications to Ramsey Theory}
In addition to Ramsey's Theorem our methods can be used to prove some well-known results in infinite-dimensional Ramsey theory. 

\begin{defn}[Galvin-Prikry, \cite{Galvin-Prikry}]
$\mathcal{X}\subseteq[\mathbb{N}]^{\infty}$ is \emph{Ramsey} if for each $[s,X]\not=\emptyset$ there exists $Y\in[s,X]$ such that either $[s,Y]\subseteq \mathcal{X}$ or $[s,Y]\cap \mathcal{X}=\emptyset$. In the notation of local Ramsey theory, $\mathcal{X}$ is Ramsey  if and only if $\mathcal{X}$ is $[\mathbb{N}]^{\infty}$-Ramsey.
\end{defn}

\begin{lem}\label{all beta}
If $\mathcal{X}$ is $\beta$-Ramsey for all nonstandard hypernatural numbers $\beta$ then $\mathcal{X}$ is Ramsey.
\end{lem}
\begin{proof}
Let $\mathcal{X}$ be given and $\beta$-Ramsey for all nonstandard hypernatural numbers $\beta$. Let $[s,X]\not=\emptyset$ and $\beta$ be some nonstandard hypernatural number in $\leftidx{^{*}}{X}$. By Lemma \ref{basic open to a tree} there exists a $\beta$-tree $T$ such that $s=st(T)$ and $[T]= [s,X]$. Since $\mathcal{X}$ is $\beta$-Ramsey there exists a $\beta$-tree $S\subseteq T$ such that $s=st(T)=st(S)$ and either $[S]\subseteq \mathcal{X}$ or $[S]\cap \mathcal{X} = \emptyset$. By Lemma \ref{diag lem} there exists $Y\in [S]$ such that $[s,Y]\subseteq [S]$. Thus, either $[s,Y]\subseteq \mathcal{X}$ or $[s,Y]\cap \mathcal{X}=\emptyset$. Since $[s,X]$ was arbitrary and $Y\in[s,X]$, $\mathcal{X}$ is Ramsey.
\end{proof}

\begin{cor}[Galvin-Prikry Theorem, \cite{Galvin-Prikry}]\label{Galvin-Prikry}
Every metrically Borel subset of $[\mathbb{N}]^{\infty}$ is Ramsey.
\end{cor}
\begin{proof}
Note that the collection of sets that are $\beta$-Ramsey for all nonstandard hypernatural numbers $\beta$ is a $\sigma$-algebra as it obtained by intersecting $\sigma$-algebras. By Lemma \ref{basic open to a tree}, for all $s\in[\mathbb{N}]^{<\infty}$ and all nonstandard hypernatural numbers $\beta$, there is a $\beta$-tree $T$ such that $[T]=[s,\mathbb{N}]=[s]$. By the $\beta$-Ellentuck theorem, $[s]$ is $\beta$-Ramsey for all nonstandard hypernatural numbers $\beta$ as $[s]=[T]$ is $\beta$-open. Since the collection of sets that are $\beta$-Ramsey for all nonstandard hypernatural numbers $\beta$ is a $\sigma$-algebra every metrically Borel set is $\beta$-Ramsey for all nonstandard hypernatural numbers $\beta$. By Lemma \ref{all beta}, every metrically Borel set is Ramsey. 
\end{proof}

\begin{cor}[Silver Theorem, \cite{silver}]
Every metrically analytic subset of $[\mathbb{N}]^{\infty}$ is Ramsey.
\end{cor}
\begin{proof}
The collection of metrically analytic sets is obtained by closing the set of metrically Borel sets under the Souslin operation. Note that the for any topological space the corresponding collection of sets with the Baire property with respect to the space is closed under the Souslin operation. In particular, for all nonstandard hypernatural numbers $\beta$, the collection of $\beta$-Ramsey sets are closed under the Souslin operation as the $\beta$-Ellentuck Theorem implies that the $\beta$-Ramsey sets coincide with the sets with the $\beta$-Baire property. Note that in the proof of the previous Corollary we showed that every metrically Borel subset of $[\mathbb{N}]^{\infty}$ is $\beta$-Ramsey for all nonstandard hypernatural numbers $\beta$. So if $\mathcal{X}$ is the image of some collection of metrically Borel subsets of $[\mathbb{N}]^{\infty}$ under the Souslin operation, then $\mathcal{X}$ is $\beta$-Ramsey for all nonstandard hypernatural numbers $\beta$. By Lemma \ref{all beta}, every metrically analytic subset of $[\mathbb{N}]^{\infty}$ is Ramsey.
\end{proof}
 
\subsection{Ultra-Ramsey theory} 
Under the assumption of the $\mathfrak{c}^{+}$-enlarging property, a saturation principle, we show that the $\vec{\alpha}$-Ellentuck Theorem is equivalent to Ultra-Ellentuck Theorem introduced by Todorcevic in \cite{RamseySpaces}.
In a footnote within \cite{AlphaTheory}, Benci and Di Nasso mention that the Alpha-Theory can be generalized to nonstandard arguments which use some prescribed level of $\kappa$-saturation. In our context, we will only need the following saturation property. 
\begin{defn}[$\mathfrak{c}^{+}$-enlarging property] Suppose $\mathcal{F}\subseteq \wp(A)$ is a family of subsets of some set $A$ and $|\mathcal{F}|\le\mathfrak{c}$. If $\mathcal{F}$ has the finite intersection property, then $$\bigcap_{F\in\mathcal{F}}\leftidx{^{*}}{F} \not=\emptyset.$$
\end{defn} In the setup of the Ultra-Ellentuck Theorem in \cite{RamseySpaces}, a sequence $\left<\mathcal{U}_{s} : s\in \mathbb{N}]^{<\infty} \right>$ of ultrafilters on $\mathbb{N}$ are chosen and all definitions and results are taken with respect to this sequence. Recall that an \emph{ultrafilter $\mathcal{U}$ on $\mathbb{N}$} is a subset of $\wp(\mathbb{N})$ such that for all subsets $A$ and $B$ of $\mathbb{N}$,
\begin{multicols}{2}
\begin{enumerate}
\item $\emptyset \not \in \mathcal{U} \ \& \ \mathbb{N}\in \mathcal{U}$,
\item $A\cup B \in \mathcal{U} \Leftrightarrow A\in U \mbox{ or } B\in U$,
\item $A \cap B \in \mathcal{U} \Leftrightarrow A \in U \ \& \ B\in U$,
\item $A\in \mathcal{U} \Leftrightarrow \mathbb{N}\setminus A \not \in \mathcal{U}$.
\end{enumerate}
\end{multicols}
An ultrafilter $\mathcal{U}$ is \emph{non-principal} if contains no finite set. The $\mathfrak{c}^{+}$-enlarging property is needed to provide a correspondence between ultrafilters on $\mathbb{N}$ and hypernatural numbers. 

\begin{prop} Suppose that the $\mathfrak{c}^{+}$-enlarging property holds. $\mathcal{U}$ is an ultrafilter on $\mathbb{N}$ if and only if there exists $\beta \in \leftidx{^{*}}{\mathbb{N}}$ such that 
$$\mathcal{U} = \{ A\subseteq \mathbb{N} : \beta \in \leftidx{^{*}}{A}\}.$$ Moreover, $\mathcal{U}$ is non-principal if and only if $\beta \not \in \mathbb{N}$.
\end{prop}
\begin{proof} If $\beta$ is a hypernatural number then Proposition \ref{star transform} (6), (7) and (8) immediately show that $\{ A\subseteq \mathbb{N} : \beta \in \leftidx{^{*}}{A}\}$ satisfies (2), (3) and (4) in the definition of an ultrafilter. Clearly, $\beta\not \in \leftidx{^{*}}{\emptyset}=\emptyset$. By assumption, $\beta \in \leftidx{^{*}}{\mathbb{N}}$. Thus $\{A\subseteq \mathbb{N} : \beta \in \leftidx{^{*}}{A}\}$ is an ultrafilter on $\mathbb{N}$.

If $\mathcal{U}$ is an ultrafilter on $\mathbb{N}$ then $|\mathcal{U}|\le\mathfrak{c}$ and $\mathcal{U}$ has the finite intersection property. By the $\mathfrak{c}^{+}$-enlarging property, $\bigcap_{A\in \mathcal{U}} \leftidx{^{*}}{A}\not=\emptyset.$ Let $\beta$ be any element of this intersection. If $A\in \mathcal{U}$ then $ \bigcap_{B\in \mathcal{U}} \leftidx{^{*}}{B}\subseteq \leftidx{^{*}}{A}.$ So for all $A\in \mathcal{U}$,  $\beta \in \leftidx{^{*}}{A}$. On the other hand, for each $A\in \wp(\mathbb{N})$, $A\cup(\mathbb{N} \setminus A)= \mathbb{N} \in \mathcal{U}$. Thus $A\in \mathcal{U}$ or $(\mathbb{N}\setminus A)\in \mathcal{U}$. Since $\beta \not \in \leftidx{^{*}}{\mathbb{N}} \setminus \leftidx{^{*}}{A}=\leftidx{^{*}}{(\mathbb{N}\setminus A)}$, $\mathbb{N}\setminus A\not \in \mathcal{U}$. Hence, $A\in \mathcal{U}$. Therefore $\{ A\subseteq \mathbb{N} : \beta \in \leftidx{^{*}}{A}\}=\mathcal{U}$.

Suppose that $\mathcal{U}$ is principal. Then there exists a finite set $\{a_{0}, a_{1}, \dots , a_{n}\}$ such that $\beta \in\leftidx{^{*}}{\{a_{0}, a_{1}, \dots , a_{n}\}}.$ By the pair axiom $$\beta \in\leftidx{^{*}}{\{a_{0}, a_{1}, \dots , a_{n}\}}=\{a_{0}, a_{1}, \dots , a_{n}\}.$$ So there exists $i\le n$ such that $\beta = a_{i}$. In particular, $\beta\in \mathbb{N}$. 

If $\beta \in \mathbb{N}$ then $\beta \in \{\beta\} = \leftidx{^{*}}{\{\beta\}}$. So $\{ A\subseteq \mathbb{N} : \beta \in \leftidx{^{*}}{A}\}$ contains a finite set and $\mathcal{U}$ is principal.
\end{proof}

\begin{defn}
Let $\vec{\mathcal{U}}=\left<\mathcal{U}_{s} : s\in \mathbb{N}]^{<\infty} \right>$ be a sequence of non-principal ultrafilters on $\mathbb{N}$ indexed by $[\mathbb{N}]^{\infty}$. A tree $T$ on $\mathbb{N}$ with stem $st(T)$ is a \emph{$\vec{\mathcal{U}}$-tree} if for all $s\in T/st(T)$
$$ \{n\in\mathbb{N} : s\cup \{n\} \in T\} \in \mathcal{U}_{s}.$$
\end{defn}

\begin{prop}\label{U-tree alpha-tree} Suppose that the $\mathfrak{c}^{+}$-enlarging property holds.  For all sequences $\vec{\mathcal{U}}=\left<\mathcal{U}_{s} : s\in [\mathbb{N}]^{<\infty} \right>$ of non-principal ultrafilters on $\mathbb{N}$ there exists a sequence $\vec{\alpha}=\left<\alpha_{s} : s\in[\mathbb{N}]^{<\infty}\right>$ of nonstandard hypernatural numbers such that for all trees $T$ on $\mathbb{N}$, $T$ is an $\vec{\mathcal{U}}$-tree if and only if $T$ is an $\vec{\alpha}$-tree.
\end{prop}
\begin{proof}
The previous proposition, allows us to chose a sequence $\left<\alpha_{s} : s\in \mathbb{N}]^{<\infty} \right>$ of nonstandard hypernatural numbers such that for all $s\in [\mathbb{N}]^{<\infty}$, $\mathcal{U}_{s} = \{ A\subseteq \mathbb{N} : \alpha_{s} \in \leftidx{^{*}}{A}\}.$ In particular, for all trees $T$ with stem $st(T)$, for all $s\in T/st(T)$,
 $$\{n\in\mathbb{N} : s\cup \{n\} \in T\} \in \mathcal{U}_{s} \iff s\cup\{\alpha_{s}\}\in \leftidx{^{*}}{T}.$$
\end{proof}

\begin{cor} Suppose that the $\mathfrak{c}^{+}$-enlarging property holds and $\vec{\mathcal{U}}=\left<\mathcal{U}_{s} : s\in [\mathbb{N}]^{<\infty} \right>$ is a sequence of non-principal ultrafilters on $\mathbb{N}$. For all $\mathcal{X}\subseteq [\mathbb{N}]^{\infty}$ and for all $\vec{\mathcal{U}}$-trees $T$ there exists an $\vec{\mathcal{U}}$-tree $S\subseteq T$ with $st(S)=st(T)$ such that one of the following holds:
\begin{enumerate}
\item $[S]\subseteq \mathcal{X}$.
\item $[S]\cap \mathcal{X}=\emptyset$.
\item For all $\vec{\mathcal{U}}$-trees $S'$, if $S'\subseteq S$ then $[S']\not \subseteq \mathcal{X}$ and $[S']\cap\mathcal{X} \not = \emptyset$.
\end{enumerate}
\end{cor}
\begin{proof}
By Proposition \ref{U-tree alpha-tree} there exists a sequence $\vec{\alpha}=\left<\alpha_{s} : s\in[\mathbb{N}]^{<\infty}\right>$ of nonstandard hypernatural numbers such that for all trees $T$ on $\mathbb{N}$, $T$ is an $\vec{\mathcal{U}}$-tree if and only if $T$ is an $\vec{\alpha}$-tree. This result is simply a restatement of Theorem \ref{alpha Ramsey theorem} using $\vec{\mathcal{U}}$-trees instead of $\vec{\alpha}$-trees.
\end{proof}

\begin{defn}Suppose that $\vec{\mathcal{U}}=\left<\mathcal{U}_{s} : s\in [\mathbb{N}]^{<\infty} \right>$ is a sequence of non-principal ultrafilters on $\mathbb{N}$.
$\mathcal{X}\subseteq [\mathbb{N}]^{\infty}$ is said to be \emph{$\vec{\mathcal{U}}$-Ramsey} if for all $\vec{\mathcal{U}}$-trees $T$ there exists an $\vec{\mathcal{U}}$-tree $S\subseteq T$ with $st(S)=st(T)$ such that either $[S]\subseteq \mathcal{X}$ or $[S]\cap \mathcal{X}=\emptyset$. $\mathcal{X}$ is said to be \emph{$\vec{\mathcal{U}}$-Ramsey null} if for all $\vec{\mathcal{U}}$-trees $T$ there exists an $\vec{\mathcal{U}}$-tree $S\subseteq T$ with $st(S)=st(T)$ such that $[S]\cap \mathcal{X}=\emptyset$.

The \emph{$\vec{\mathcal{U}}$-Ellentuck space} is the topological space on $[\mathbb{N}]^{\infty}$ generated by $\{[T]:\mbox{$T$ is a $\vec{\mathcal{U}}$-tree}\}$. $\mathcal{X}\subseteq [\mathbb{N}]^{\infty}$ \emph{is $\vec{\mathcal{U}}$-nowhere dense/ is $\vec{\mathcal{U}}$-meager/ has the $\vec{\mathcal{U}}$-Baire property} if it is nowhere dense/ is meager/ has the Baire property with respect to the $\vec{\mathcal{U}}$-Ellentuck space. 

A triple $([\mathbb{N}]^{\infty},\vec{\mathcal{U}}, \subseteq)$ is said to be an \emph{ultra-Ramsey space} if the collection of $\vec{\mathcal{U}}$-Ramsey sets coincides with the $\sigma$-algebra of sets with the $\vec{\mathcal{U}}$-Baire property and the collection of $\vec{\mathcal{U}}$-Ramsey null sets coincides with the $\sigma$-ideal of $\vec{\mathcal{U}}$-meager sets.
\end{defn}

\begin{cor}[The Ultra-Ellentuck Theorem, Todorcevic \cite{RamseySpaces}] Suppose that the $\mathfrak{c}^{+}$-enlarging property holds and $\vec{\mathcal{U}}=\left<\mathcal{U}_{s} : s\in [\mathbb{N}]^{<\infty} \right>$ is a sequence of non-principal ultrafilters on $\mathbb{N}$. Then $([\mathbb{N}]^{\infty},\vec{\mathcal{U}}, \subseteq)$ is an ultra-Ramsey space.
\end{cor}
\begin{proof}
By Proposition \ref{U-tree alpha-tree} there exists a sequence $\vec{\alpha}=\left<\alpha_{s} : s\in[\mathbb{N}]^{<\infty}\right>$ of nonstandard hypernatural numbers such that for all trees $T$ on $\mathbb{N}$, $T$ is an $\vec{\mathcal{U}}$-tree if and only if $T$ is an $\vec{\alpha}$-tree. This result is simply a restatement of the $\vec{\alpha}$-Ellentuck Theorem in the setting of $\vec{\mathcal{U}}$-trees instead of $\vec{\alpha}$-trees.
\end{proof}

\subsection{Selective ultrafilters} Suppose that $\mathcal{U}$ is a non-principal ultrafilter on $\mathbb{N}$. In this subsection we consider the ultra-Ramsey theory in the context of sequences $\vec{\mathcal{U}}=\left<\mathcal{U}_{s}:s\in[\mathbb{N}]^{<\infty}\right>$ such that for all $s\in [\mathbb{N}]^{<\infty}$, $\mathcal{U}_{s}=\mathcal{U}$. Note that we do not suppress the arrow as in \cite{RamseySpaces}. 

Concepts like $\vec{\mathcal{U}}$-Ramsey and $\mathcal{U}$-Ramsey are different with the latter definition coming from local Ramsey theory and the former from ultra-Ramsey theory. As we shall see these two notions coincide precisely when the ultrafilter $\mathcal{U}$ is taken to be a selective ultrafilter.

\begin{defn}
An ultrafilter $\mathcal{U}$ on $\mathbb{N}$ is \emph{selective} if for all sequences $X_{0}\supseteq X_{1}\supseteq X_{2} \supseteq \cdots$ of sets in $\mathcal{U}$, there exists $X=\{x_{0},x_{1},\dots\}\in\mathcal{U}$ enumerated in increasing order such that for all $n\in\mathbb{N},$   $x_{n+1}\in X_{n}.$
\end{defn}

\begin{thm}\label{selective ultrafilter theorem}
Suppose that the $\mathfrak{c}^{+}$-enlarging property holds and $\mathcal{U}$ is selective ultrafilter on $\mathbb{N}$. Let $\mathcal{X}\subseteq [\mathbb{N}]^{\infty}$. The following are equivalent:
\begin{enumerate}
\item $\mathcal{X}$ has the $\vec{\mathcal{U}}$-Baire property.
\item $\mathcal{X}$ is $\vec{\mathcal{U}}$-Ramsey.
\item $\mathcal{X}$ has the $\mathcal{U}$-Baire property.
\item $\mathcal{X}$ is $\mathcal{U}$-Ramsey.
\end{enumerate}
Furthermore, the following are equivalent:
\begin{enumerate}
\item[(5)] $\mathcal{X}$ is $\vec{\mathcal{U}}$-meager.
\item[(6)] $\mathcal{X}$ is $\vec{\mathcal{U}}$-Ramsey null.
\item[(7)] $\mathcal{X}$ is $\mathcal{U}$-meager.
\item[(8)] $\mathcal{X}$ is $\mathcal{U}$-Ramsey null.
\end{enumerate}
\end{thm}
\begin{proof}
Suppose that $\mathcal{U}$ is a selective ultrafilter on $\mathbb{N}$. Note that (1)$\iff$(2) and (5)$\iff$(6) follow from the Ultra-Ellentuck Theorem. By Lemma \ref{U-tree alpha-tree} there exists a nonstandard hypernatural number $\beta$ such that $\mathcal{U} =\{ X\in[\mathbb{N}]^{\infty} : \beta\in\leftidx{^{*}}{X}\}$. Since $\mathcal{U}$ is selective, Theorem \ref{beta-SCIP} implies that if $T$ is a $\beta$-tree then for all $s\in T/st(T)$ there exists $X_{s}\in[s,\mathbb{N}]$ such that $\beta\in\leftidx{^{*}}{X_{s}}$ and $[s,X_{s}] \subseteq [T].$ Thus for each $\beta$-tree $T$, 
$$[T]=\bigcup_{s\in T/st(T)} [s, X_{s}].$$
 In particular, $\{[s,X] : \beta\in\leftidx{^{*}}{X}\}$ generates the $\beta$-Ellentuck topology. Since the $\beta$-Ellentuck topology and the $\mathcal{U}$-Ellentuck topology are identical we have (1)$\iff$(3) and (5)$\iff$(7).

(2)$\implies$(4). Suppose that $\mathcal{X}$ satisfies (2) and $[s,X]\not=\emptyset$ with $\beta\in\leftidx{^{*}}{X}$. Note that $\mathcal{X}$ is $\beta$-Ramsey. By Lemma \ref{basic open to a tree} there exists a $\beta$-tree $T$ such that $st(T)=s$ and $[T]=[s,X]$. The $\beta$-Ellentuck Theorem implies that there exists a $\beta$-tree $S\subseteq T$ such that $st(S)=st(T)=s$ and either $[S]\subseteq \mathcal{X}$ or $[S]\cap \mathcal{X}=\emptyset$. Since $\mathcal{U}$ is selective, Theorem \ref{beta-SCIP} implies that there exists $Y\in[s,\mathbb{N}]$ such that $\beta\in\leftidx{^{*}}{Y}$ and $[s,Y] \subseteq [S]$. Hence, $Y\in[s,X]$ and either $[s,Y]\subseteq \mathcal{X}$ or $[s,Y]\cap \mathcal{X} =\emptyset$. Since $[s,X]$ was arbitrary (4) holds. By a similar argument (6)$\implies$(8).

(4)$\implies$(2). Suppose that $\mathcal{X}$ satisfies (4) and let $T$ be $\vec{\mathcal{U}}$-tree. Note that $T$ is a $\beta$-tree. Since $\mathcal{U}$ is selective, Theorem \ref{beta-SCIP} implies that there exists $X\in[st(T),\mathbb{N}]$ such that $\beta\in\leftidx{^{*}}{X}$ and $[st(T),X]\subseteq[T]$. By (4) there exists $Y\in[st(T),X]$ such that $\beta\in\leftidx{^{*}}{Y}$ and either $[st(T),Y]\subseteq \mathcal{X}$ or  $[st(T),Y]\cap\mathcal{X} =\emptyset$. By Lemma \ref{basic open to a tree} there exists a $\beta$-tree $S$ such that $st(S)=st(T)$ and $[S]=[s,Y]$. Thus $S\subseteq T$, $st(T)=st(S)$, $S$ is a $\vec{\mathcal{U}}$-tree and either $[S]\subseteq \mathcal{X}$ or $[S]\cap \mathcal{X}=\emptyset$. Since $T$ was arbitrary (2) holds. By a similar argument (8)$\implies$(6).

Altogether, we have shown that (3)$\iff$(1)$\iff$(2)$\iff$(4) and (7)$\iff$(5)$\iff$(6)$\iff$(8).
\end{proof}

The equivalences (3)$\iff$(4) and (7)$\iff$(8) in the previous theorem follow from the work of Louveau in \cite{Louveau} but in a different framework. For a proof in the setting of ultra-Ramsey theory see Corollary 7.24 in \cite{RamseySpaces}. 

\begin{cor}[Louveau, \cite{Louveau}]\label{Louveau} If $\mathcal{U}$ is selective ultrafilter on $\mathbb{N}$ then $([\mathbb{N}]^{\infty}, \mathcal{U}, \subseteq)$ is a topological Ramsey space.
\end{cor}
\begin{proof}
Note that the string of implications in the previous proof established (3)$\iff$(4) and (7)$\iff$(8). These two equivalences show that the triple satisfies the definition of a topological Ramsey space.
\end{proof}

\subsection{The strong Cauchy infinitesimal principle}
Cleave in \cite{Cleave} proposed an interpretation within nonstandard analysis of the infinitely small quantities described by Cauchy as ``variables converging to zero" in his nineteenth century textbooks. Benci and Di Nasso in \cite{AlphaTheory} have formulated a stronger version of Cleave's interpretation which characterizes when $\{X\in[\mathbb{N}]^{\infty} : \alpha\in\leftidx{^{*}}{X}\}$ is a selective ultrafilter.
\begin{defn}(Strong Cauchy Infinitesimal Principle, SCIP)
Every nonstandard hypernatural number $\beta$ is the ideal value of an increasing sequence of natural numbers. 
\end{defn}
In \cite{AlphaTheory} Benci and Di Nasso show that the Alpha-Theory cannot prove nor disprove SCIP. Moreover, they show that the Alpha-Theory+SCIP is a ``sound" system. That is, the system does not prove a contradiction. 
\begin{thm} The following are equivalent:
\begin{enumerate}
\item The strong Cauchy infinitesimal principle.
\item $\{X\in[\mathbb{N}]^{\infty} : \alpha\in\leftidx{^{*}}{X}\}$ is a selective ultrafilter.
\item If $T$ is an $\alpha$-tree and $s\in T/st(T)$ then there exists $X\in[s,\mathbb{N}]$ such that $\alpha\in \leftidx{^{*}}{X}$ and $[s,X] \subseteq [T].$
\item $([\mathbb{N}]^{\infty},\{X\in[\mathbb{N}]^{\infty} : \alpha\in\leftidx{^{*}}{X}\},\subseteq)$ is a topological Ramsey space.
\end{enumerate}
\end{thm}
\begin{proof}
The equivalence (1)$\iff$(2) follows from Corollary 6.8 (ii) in \cite{AlphaTheory}. The equivalence (2)$\iff$(3) follows from Theorem \ref{beta-SCIP}. By Corollary \ref{Louveau} we have (2)$\implies$(4). We complete the proof by showing that (4)$\implies$(1).

Let $\mathcal{U}=\{X\in[\mathbb{N}]^{\infty} : \alpha\in\leftidx{^{*}}{X}\}$ and suppose that $([\mathbb{N}]^{\infty},\mathcal{U},\subseteq)$ is a topological Ramsey space. Let $\varphi=\left<\varphi_{i} :i\in\mathbb{N}\right>$ be a sequence of natural numbers such that $\varphi[\alpha]$ is a nonstandard hypernatural number. Consider the following sets 
$$A_{\varphi}=\{\{i,j\}\in[\mathbb{N}]^{2} : i<j \ \& \ \varphi_{i}<\varphi_{j}\}$$
$$\mathcal{X}_{\varphi} =\{X\in [\mathbb{N}]^{\infty}: r_{2}(X)\in A_{\varphi}\}$$
here $r_{2}(X)$ denotes the two smallest elements of $X$. It is clear that $\mathcal{X}_{\varphi}$ is metrically open. So it is also $\mathcal{U}$-open. Since $([\mathbb{N}]^{\infty},\mathcal{U},\subseteq)$ is a topological Ramsey space, there exists $X\in\mathcal{U}$ such that $[\emptyset,X]\subseteq \mathcal{X}_{\varphi}$ or $[\emptyset, X]\cap \mathcal{X}_{\varphi}$.

Notice that $[\emptyset,X]\cap \mathcal{X}_{\varphi}=\emptyset$ is impossible as it implies that for all $i\in X$, $i>\min(X)\implies\{\min(X),i\}\not\in A_{\varphi}$. That is, for all $i\in X$, $\varphi_{\min(X)}\ge \varphi_{i}$. Since $\alpha\in\leftidx{^{*}}{X}$, $\varphi_{\min(X)}\ge \varphi[\alpha]$. However, this contradicts the fact that $\varphi[\alpha]$ is a nonstandard hypernatural number.  Thus, $[\emptyset,X]\subseteq \mathcal{X}_{\varphi}$.

Let $\{i_{0},i_{1},i_{2},\dots\}$ be an increasing enumeration of $X$ and define $\psi=\left<\psi_{i} :i\in \mathbb{N}\right>$ as follows
$$\psi_{i} =\begin{cases} \varphi_{i_{n-1}} & \text{if } i_{n-1}\le i \le i_{n}, \\
0 & \text{if } i <i_{0}.\end{cases}$$

By construction, for all $i\in X$, $\varphi_{i}=\psi_{i}$. Since $\alpha\in\leftidx{^{*}}{X}$, $\varphi[\alpha]=\psi[\alpha]$. Since $\psi$ is an increasing sequence of natural numbers SCIP holds.
\end{proof}

\section{Abstract Alpha-Ramsey Theory}

We extend the main results of the previous section to the setting of triples $(\mathcal{R},\le,r)$ where $\le$ is a quasi-order on $\mathcal{R}$ and $r$ is a function with domain $\mathbb{N}\times \mathcal{R}$. The prototype example of such a triple is $([\mathbb{N}]^{\infty}, \subseteq, r)$ where $r$ is the map such that for all $n\in\mathbb{N}$ and for all $X=\{x_{0},x_{1}, x_{2}, \dots\}$, listed in increasing order,
$$ r(n, X )  = \begin{cases} \emptyset & \mbox{ if } n=0, \\
                                                        \{x_{0},\dots,x_{n-1}\} & \mbox{ otherwise}.
																					 \end{cases}$$
For example, if $E$ is the set of even numbers then $r(3,E) = \{2,4,6\}$. For this triple the range of $r$ is $[\mathbb{N}]^{<\infty}$ and for all $s\in[\mathbb{N}]^{<\infty}$ and for all $X\in [\mathbb{N}]^{\infty}$, $s\sqsubseteq X$ if and only if there exists $n\in\mathbb{N}$ such that $r(n,X)=s$. 

For the remainder of this section we fix a triple $(\mathcal{R},\le,r)$ where $\le$ is a quasi-order on $\mathcal{R}$ and $r$ is a function with domain $\mathbb{N}\times \mathcal{R}$. For $n\in \mathbb{N}$ and $X\in \mathcal{R}$, we abbreviate $r(n,X)$ by $r_{n}(X)$ and call it the \emph{$n^{th}$-approximation of $X$}. The elements of the range of $r$, $\{r_{n}(X) : n\in\mathbb{N} \ \& \ X\in \mathcal{R}\}$, are called \emph{finite approximations of $\mathcal{R}$}. The set of finite approximations, i.e. the range of $r$, is denoted by $\mathcal{AR}$. For $n\in \mathbb{N}$ and $X\in \mathcal{R}$ we use the following notation 
$$\mathcal{AR}_{n} =\{r_{n}(X) \in \mathcal{AR} : X\in \mathcal{R}\},$$
$$\mathcal{AR}_{n}\restriction X =\{r_{n}(Y) \in \mathcal{AR} : Y\in \mathcal{R} \ \& \ Y \le X\},$$
$$\mathcal{AR} \restriction X =\bigcup_{n=0}^{\infty} \mathcal{AR}_{n}\restriction X.$$

If $s\in\mathcal{AR}$ and $X\in \mathcal{R}$ then we say \emph{$s$ is an initial segment of $X$} and write $s\sqsubseteq X$, if there exists $n\in\mathbb{N}$ such that $s = r_{n}(X)$. If $s\sqsubseteq X$ and $s\not=X$ then we write $s\sqsubset X$. We use the following notation:
$$[s]= \{ Y \in \mathcal{R} : s\sqsubseteq Y\},$$
$$[s,X] =\{Y \in \mathcal{R} : s \sqsubseteq Y \le X\}.$$
\begin{rem}
From this point forward, in order to avoid any trivial cases, we assume that $(\mathcal{R},\le,r)$ has the property that for all $[s,X]\not=\emptyset$, $\{ t\in \mathcal{AR} : |t|=|s|+1,\ s\sqsubseteq t\sqsubseteq X\}$ is infinite. 
\end{rem}

\begin{defn}
Suppose that $\mathcal{C}\subseteq \mathcal{R}$. $\mathcal{X}\subseteq \mathcal{R}$ is \emph{$\mathcal{C}$-Ramsey} if for all $[s,X]\not=\emptyset$ with $X\in\mathcal{C}$ there exists $Y\in[s,X]\cap \mathcal{C}$ such that either $[s,Y]\subseteq \mathcal{X}$ or $[s,Y]\cap \mathcal{X}=\emptyset.$ $\mathcal{X}\subseteq \mathcal{R}$ is \emph{$\mathcal{C}$-Ramsey null} if for all $[s,X]\not=\emptyset$ with $X\in\mathcal{C}$ there exists $Y\in[s,X]\cap\mathcal{C}$ such that $[s,Y]\cap \mathcal{X}=\emptyset.$
\end{defn}

\begin{defn} Suppose that $\mathcal{C}\subseteq \mathcal{R}$. We say that $(\mathcal{R}, \mathcal{C}, \le,r)$ is a \emph{topological Ramsey space} if the following conditions hold:
\begin{enumerate}
\item $\{[s,X] :X\in\mathcal{C}\}$ is a basis for a topology on $\mathcal{R}$.
\item The collection of $\mathcal{C}$-Ramsey sets coincides with the $\sigma$-algebra of sets with the Baire property with respect to the topology generated by $\{[s,X] :X\in\mathcal{C}\}$.
\item The collection of $\mathcal{C}$-Ramsey null sets coincides with the $\sigma$-ideal of meager sets with respect to the topology generated by $\{[s,X] :X\in\mathcal{C}\}$.
\end{enumerate}
If $\mathcal{C}=\mathcal{R}$ then we abbreviate $(\mathcal{R},\mathcal{C},\le,r)$ by $(\mathcal{R},\le,r)$ and we omit the $\mathcal{C}$ from the above definitions. 

For example, a subset ${\mathcal X}$ of $\mathcal{R}$ is \emph{Ramsey} if for every $\emptyset \not= [s,X],$ there is a $Y \in [s,X]$ such that $[s,Y] \subseteq {\mathcal X}$ or $[s,Y]\cap {\mathcal X} = \emptyset.$ A subset ${\mathcal X}$ of ${\mathcal R}$ is \emph{Ramsey null} if for every $\emptyset \not= [s,X],$ there is a $Y \in [s,X]$ such that $[s.Y]\cap {\mathcal X} = \emptyset.$
\end{defn}

Abstract local Ramsey theory is concerned with characterizing the conditions on $\mathcal{C}$ which guarantee that $(\mathcal{R},\mathcal{C},\le,r)$ forms a topological Ramsey space.
\subsection{The Abstract Ellentuck Theorem}
We follow the presentation of the Abstract Ellentuck Theorem given by Todorcevic in \cite{RamseySpaces}, rather than the earlier reference \cite{CarlsonSimpson}. In particular, we introduce four axioms about triples $({\mathcal R}, \le, r )$ sufficient for proving an abstract version of the Ellentuck Theorem. The first axiom we consider tells us that $\mathcal{R}$ is collection of infinite sequences of objects and $\mathcal{AR}$ is collection of finite sequences approximating these infinite sequences. 

\begin{axm}(A.1 - Sequencing) For each $X,Y \in {\mathcal R}$, \begin{enumerate}
		\item[(a)] $r_{0}(X) = \emptyset$.
		\item[(b)] $X\not = Y$ implies $r_{i}(X) \not = r_{i}(X)$ for some $i\in\mathbb{N}$.
		\item[(c)] $r_{i}(X) = r_{j}(Y)$ implies $i=j$ and $r_{k}(X) = r_{k}(Y)$ for all $k<i$.
	\end{enumerate}
\end{axm}
On the basis of this axiom, $\mathcal{R}$ can be identified with a subset of $\mathcal{AR}^{\mathbb{N}}$ by associating $X\in\mathcal{R}$ with the sequence $\left<r_{i}(X)): i\in\mathbb{N}\right>$. Similarly, $s\in\mathcal{AR}$ can be identified with $\left<r_{i}(X)):i<j\right>$ where $j$ is the unique natural number such that $s=r_{j}(X)$ for some $X\in\mathcal{R}$. For each $s\in \mathcal{AR}$, let $|s|$ equal the natural number $i$ for which $s=r_{i}(s)$. For $s,t \in \mathcal{AR}$, $s \sqsubseteq t$ if and only if $s=r_{i}(t)$ for some $i \le |t|$. $s \sqsubset t$ if and only if $s=r_{i}(t)$ for some $i<|t|$.

\begin{axm}(A.2 - Finitization) There is a quasi-ordering $\le_{\mathrm{fin}}$ on $\mathcal{AR}$ such that
	\begin{enumerate}
		\item[(a)] $\{t \in \mathcal{AR} : t \le_{\mathrm{fin}} s\}$ is finite for all $s \in \mathcal{AR}$,
		\item[(b)] $ X \le Y$ iff $(\forall i)(\exists j) \ r_{i}(X) \le_{\mathrm{fin}} r_{j}(Y),$
		\item[(c)] $\forall s,t,u \in \mathcal{AR} [ s \sqsubseteq t \wedge t \le_{\mathrm{fin}} u \rightarrow \exists v \sqsubseteq u \ s \le_{\mathrm{fin}} v ].$
	\end{enumerate}
For $s\in\mathcal{AR}$ and $X\in\mathcal{R}$, $\mathrm{depth}_{X}(s)$ is the least $i$, if it exists, such that $ s \le_{\mathrm{fin}} r_{i}(X)$. If such an $i$ does not exist, then we write $\mathrm{depth}_{X}(s) = \infty$. If depth$_{X}(s)=i< \infty,$ then $[\mathrm{depth}_{X}(s),X]$ denotes $[r_{i}(X), X]$.  
\end{axm}
\begin{axm}(A.3 - Amalgamation) For each $X, Y\in\mathcal{ R}$ and each $s\in\mathcal{AR}$,
	\begin{enumerate}
		\item[(a)] If $\mathrm{depth}_{X}(s) < \infty$ then $[s,X]\not = \emptyset$ for all $X \in [\mathrm{depth}_{Y}(s),Y]$.
		\item[(b)] $X \le Y$ and $[s, X]\not = \emptyset$ imply that there is an $X' \in [ \mathrm{depth}_{Y}(s),Y]$ such that $\emptyset \not = [s, X'] \subseteq [s, X]$.
	\end{enumerate}
\end{axm}
If $n>|s|$, then $r_{n}[s,X]$ denotes the collection of all $t\in \mathcal{AR}_{n}$ such that $ s \sqsubset t$ and $t \le_{\mathrm{fin}} X$.

\begin{axm}(A.4 - Pigeonhole Principle) For each $Y\in\mathcal{R}$ and each $s\in\mathcal{AR}$, if depth$_{Y}(s) < \infty$ and $\mathcal{ O} \subseteq \mathcal{AR}_{|s|+1}$, then there is $X \in [\mathrm{depth}_{Y}(s), Y]$ such that
$$ r_{|s|+1}[s,X] \subseteq {\mathcal O}\mbox{ or } r_{|s|+1}[s, X] \subseteq {\mathcal O}^{c}.$$
\end{axm}
The next result, using a slightly different set of axioms, is a theorem of Carlson and Simpson in \cite{CarlsonSimpson}. The version using {A.1}-{A.4} below follows from the work of Todorcevic in \cite{RamseySpaces}. In the next theorem the topology on $\mathcal{AR}^{\mathbb{N}}$ refers to the product topology where $\mathcal{AR}$ is endowed with the discrete topology. 
\begin{thm}[Abstract Ellentuck Theorem, \cite{RamseySpaces,CarlsonSimpson}] If $(\mathcal{R}, \le , r)$ is a closed subspace of $\mathcal{AR}^{\mathbb{N}}$ satisfies \emph{A.1, A.2, A.3} and \emph{A.4} then $(\mathcal{R}, \le , r)$ forms a topological Ramsey space.
\end{thm}
\begin{example}[The Ellentuck Space] The triple $([\mathbb{N}]^{\infty}, \subseteq, r)$ where $r$ is the map such that for all $n\in\mathbb{N}$ and for all $X=\{x_{0},x_{1}, x_{2}, \dots\}$, listed in increasing order,
$$ r(n, X )  = \begin{cases} \emptyset & \mbox{ if } n=0, \\
                                                        \{x_{0},\dots,x_{n-1}\} & \mbox{ otherwise},
																					 \end{cases}$$
gives rise to the Ellentuck topology on $[\mathbb{N}]^{\infty}$. It is easy to show that $([\mathbb{N}]^{\infty}, \subseteq, r)$ satisfies axioms A.1-A.4 and $[\mathbb{N}]^{\infty}$ is a closed subspace of $([\mathbb{N}]^{<\infty})^{\mathbb{N}}$. Hence, the Abstract Ellentuck Theorem implies that $([\mathbb{N}]^{\infty}, \subseteq, r)$ is a topological Ramsey space. In the terminology of the first section, $([\mathbb{N}]^{\infty},[\mathbb{N}]^{\infty}, \subseteq)$ is a topological Ramsey space. All the results of this section when taken with respect to the triple $([\mathbb{N}]^{\infty}, \subseteq, r)$ reduce to the theorems of the first section.														
\end{example}

\begin{example}[The Milliken Space]
Let $\mathrm{FIN}$ denote the collection of nonempty finite subsets of $\mathbb{N}$. That is, $\mathrm{FIN}=[\mathbb{N}]^{<\infty}\setminus\{\emptyset\}$. We say that a sequence $\left<s_{i} :i\in\mathbb{N}\right>$ of elements of $\mathrm{FIN}$ is a \emph{block sequence} if for all $i\in\mathbb{N}$, $\max(s_{i})<\min(s_{i+1})$. The set of all block sequences is denoted by $[\mathrm{FIN}]^{\infty}$. A finite sequence with the same property is called a \emph{finite block sequence}. The set of all finite block sequences is denoted by $[\mathrm{FIN}]^{<\infty}$. For $S=\left<s_{i} :i\in\mathbb{N}\right>$ and $T=\left<t_{i} :i\in\mathbb{N}\right>$ in 
$[\mathrm{FIN}]^{\infty}$, we say that $S\le T$ if and only if for all $i\in\mathbb{N}$ there exists $I\in\mathrm{FIN}$ such that $s_{i} = \cup_{i\in I} t_{i}$. Let $r:\mathbb{N}\times [\mathrm{FIN}]^{\infty} \rightarrow [\mathrm{FIN}]^{<\infty}$ be the function such that for all $(n,S)\in \mathbb{N}\times [\mathrm{FIN}]^{\infty}$,
$$ r(n,S) = \left< s_{i} : i < n\right>.$$
The work of Milliken in \cite{milliken} implies that the triple $([\mathrm{FIN}]^{\infty}, \le, r)$ forms a topological Ramsey space. However, this fact also follows form the Abstract Ellentuck Theorem.

The relation $\le$ can be finitized, for $0<n\le m$ in $\mathbb{N}$, $s=\left<s_{i} :i<n\right>$ and $t=\left<t_{i} :i<m\right>$ in 
$[\mathrm{FIN}]^{<\infty}$, we say that $s\le_{\mathrm{fin}} t$ if and only if for all $i<n$ there exists $I\in\mathrm{FIN}$ such that $s_{i} = \cup_{i\in I} t_{i}$. With this quasi-order on $[\mathrm{FIN}]^{<\infty}$, the triple $([\mathrm{FIN}]^{\infty},\le, r)$ satisfies axioms A.1-A.4. For this triple A.4 is equivalent to celebrated Hindman's Theorem from \cite{Hindman}. It is easy to show that $[\mathrm{FIN}]^{\infty}$ is a closed subspace of $([\mathrm{FIN}]^{<\infty})^{\mathbb{N}}$. The Abstract Ellentuck Theorem implies $([\mathrm{FIN}]^{\infty}, \le, r)$ is a topological Ramsey space. \end{example}

For more examples of topological Ramsey spaces satisfying A.1-A.4 and some recent developments in topological Ramsey theory see \cite{DMN,dob,DT1,DT2,DMT,DM,MijaresSelective,Mijares1,Mijares2,Mijares3,Mijares4,RamseySpaces,Tr1,Tr2}.

\subsection{Abstract $\vec{\alpha}$-Ramsey Theory}
A subset $T$ of $\mathcal{AR}$ is called a \emph{tree on $\mathcal{R}$} if $T\not =\emptyset$ and for all $s,t\in\mathcal{AR}$,
$$s\sqsubseteq t\in T \implies s\in T.$$ For a tree $T$ on $\mathcal{R}$ and $n\in \mathbb{N}$, we use the following notation: 
$$[T] =\{X\in\mathcal{R} : \forall s\in\mathcal{AR}( s\sqsubseteq X \implies s \in T) \},$$
$$T(n) =\{s\in T : s\in \mathcal{AR}_{n}\}.$$
The \emph{stem of $T$}, if it exists, is the $\sqsubseteq$-maximal $s$ in  $T$ that is $\sqsubseteq$-comparable to every element of $T$. If $T$ has a stem we denote it by $st(T)$. For $s\in T$, we use the following notation $$T/s = \{ t\in T : s\sqsubseteq t\}.$$

\begin{lem} \label{infinity lemma} Suppose that for all $s\in\mathcal{AR}$, $\leftidx{^{*}}{s} =s$. Let $\beta\in\leftidx{^{*}}{(\mathcal{AR})}\setminus\mathcal{AR}$ such that $r_{|\beta|-1}(\beta)\in\mathcal{AR}$. If $\mathcal{F}\subseteq \mathcal{AR}$ and $\beta\in \leftidx{^{*}}{\mathcal{F}}$ then $\mathcal{F}$ is infinite. In particular, if $X\in\mathcal{R}$ and $\beta\in \leftidx{^{*}}{r_{|\beta|}[r_{|\beta|-1}(\beta),X]}$ then $r_{|\beta|}[r_{|\beta|-1}(\beta),X]$ is infinite.
\end{lem}
\begin{proof} Toward a contradiction suppose that $\{s_{0},s_{1},\dots,s_{n}\}$ is an enumeration of $\mathcal{F}$. Then by Axiom $\alpha$4,
$$\beta\in\leftidx{^{*}}{\mathcal{F}}=\{\leftidx{^{*}}{s_{0}},\leftidx{^{*}}{s_{1}},\dots,\leftidx{^{*}}{s_{n}}\} = \{s_{0},s_{1},\dots,s_{n}\}\subseteq \mathcal{AR}.$$ However, this is a contradiction since $\beta\not\in \mathcal{AR}$.
\end{proof}

\begin{lem}\label{abstract alpha sequence}
If $(\mathcal{R},\le, r)$ satisfies \emph{A.1}, \emph{A.2} and \emph{A.4} then for all $s\in \mathcal{AR}$ and for all $X\in\mathcal{X}$ such that $s\sqsubseteq X$, there exists $\alpha_{s}\in \leftidx{^{*}}{(\mathcal{AR}\restriction X)}\setminus (\mathcal{AR}\restriction X)$ such that
$$s \sqsubseteq \alpha_{s}\in \leftidx{^{*}}{\mathcal{AR}_{|s|+1}}.$$
\end{lem}
\begin{proof}
Let $s\in\mathcal{AR}$ and $X\in \mathcal{R}$ such that $s\sqsubseteq X$. For all $n\in\mathbb{N}$, let
$$\mathcal{O}_{n} = \{t\in r_{|s|+1}[s,X] : t \le_{\mathrm{fin}} r_{n}(X)\}.$$
By A.4 there exists a sequence $X\ge X_{1} \ge X_{2} \ge X_{3} \ge \dots$ such that for all positive natural numbers $i$, either $r_{|s|+1}[s,X_{i}]\subseteq \mathcal{O}_{i}$ or $r_{|s|+1}[s,X_{i}]\cap \mathcal{O}_{i}=\emptyset$. By A.2 (a), for all $i\in\mathbb{N}$, $\mathcal{O}_{i}$ is finite. So it can not be the case that $r_{|s|+1}[s,X_{i}]\subseteq \mathcal{O}_{i}$ (recall that, we assumed $r_{|s|+1}[s,X]$ was infinite for any $[s,X]\not=\emptyset$ to avoid trivialities). Thus for all $i\in \mathbb{N}$, $r_{|s|+1}[s,X_{i}]\cap \mathcal{O}_{i}=\emptyset$. Note that $$r_{|s|+1}[s,X_{1}]\supseteq r_{|s|+1}[s,X_{2}]\supseteq r_{|s|+1}[s, X_{3}]\supseteq\dots$$ Hence $\{r_{|s|+1}[s,X_{i}]:i\in \mathbb{N}\}$ has the finite intersection property. By the countable enlargement property,
$$\bigcap_{i=1}^{\infty} \leftidx{^{*}}{r_{|s|+1}[s,X_{i}]}\not=\emptyset.$$
Let $\alpha_{s}$ be any element of this intersection. It is clear that $s \sqsubseteq \alpha_{s} \in \leftidx{^{*}}{\mathcal{AR}\restriction X}$. 

Toward a contradiction suppose that $\alpha_{s}\in \mathcal{AR}\restriction X$. Then there exists $Y\le X$ such that $\alpha_{s} = r_{|s|+1}(Y)$. By A.2 (b) there exists $n\in\mathbb{N}$ such that $r_{|s|+1}(Y)\le_{\mathrm{fin}} r_{n}(X)$. So $r_{|s|+1}(Y)\in \mathcal{O}_{n}$ which is a contradiction since $\alpha_{s}\in \leftidx{^{*}}{r_{|s|+1}[s, X_{n}]}$ and $r_{|s|+1}[s, X_{n}]\cap \mathcal{O}_{n}=\emptyset$ . Therefore $\alpha_{s}\not \in \mathcal{AR}\restriction X$.
\end{proof}
From this point forward we will assume in each statement that $(\mathcal{R},\le, r)$ satisfies {A.1}, {A.2} and {A.4} and for all $s\in\mathcal{AR}$, $\leftidx{^{*}}{s}=s$. Because of this assumption, for the rest of this section, we can fix a sequence $\vec{\alpha}=\left< \alpha_{s} : s\in \mathcal{AR}\right>$ where for all $s\in \mathcal{AR}$ 
$$s\sqsubseteq \alpha_{s}\in \leftidx{^{*}}{\mathcal{AR}_{|s|+1}}.$$ By Lemma \ref{abstract alpha sequence} at least one such $\vec{\alpha}$ sequence exists.

\begin{defn} 
An \emph{$\vec{\alpha}$-tree} is a tree $T$ on $\mathcal{R}$ with stem $st(T)$ such that $T/st(T)\not=\emptyset$ and for all $s\in T/st(T)$, $$\alpha_{s} \in  \leftidx{^{*}}{T}.$$
\end{defn}

\begin{example}Note that $\mathcal{AR}$ is a tree on $\mathcal{R}$ with stem $\emptyset$. Moreover, for all $s\in\mathcal{AR}$, $\alpha_{s}\in \leftidx{^{*}}{\mathcal{AR}}$. Thus, $\mathcal{AR}$ is an $\vec{\alpha}$-tree.
\end{example}

\begin{lem}\label{alpha diag} Assume that $(\mathcal{R},\le, r)$ satisfies \emph{A.1}, \emph{A.2} and \emph{A.4} and for all $s\in\mathcal{AR}$, $\leftidx{^{*}}{s}=s$. Suppose that $H\subseteq \mathcal{AR}$ and for all $s\in H$, $\alpha_{s}\in \leftidx{^{*}}{H}$. Then for all $\vec{\alpha}$-trees $T$, if $st(T)\in H$ then there exists an $\vec{\alpha}$-tree $S\subseteq T$ with $st(S)=st(T)$ such that $S/st(S)\subseteq H$.
\end{lem}
\begin{proof}
Let $H\subseteq \mathcal{AR}$ such that for all $s\in H$, $\alpha_{s}\in \leftidx{^{*}}{H}$. Suppose that $T$ is an $\vec{\alpha}$-tree and $st(T)\in H$. We construct an $\vec{\alpha}$-tree $S$, level-by-level, recursively as follows  
$$\begin{cases}
L_{0} =\{st(T)\} \\
L_{n+1} = \{ t \in \mathcal{AR}: \exists s\in L_{n}, \ s\sqsubseteq t, \ |t|=|s|+1 \ \& \ t\in H\cap T\}.
\end{cases}$$
 Since $T$ is an $\vec{\alpha}$-tree, for all $s\in L_{n}, \ s\sqsubseteq \alpha_{s}, \ |\alpha_{s}|=|s|+1$ and $\alpha_{s}\in\leftidx{^{*}}{H}\cap\leftidx{^{*}}{T}=\leftidx{^{*}}{(H\cap T)}$. In particular, $\alpha_{s}\in \leftidx{^{*}}{L_{n+1}}$.  Let 
$$S = \{s \in \mathcal{AR} : s \sqsubseteq st(T)\} \cup \bigcup_{n=0}^{\infty} L_{n}.$$

It is clear that $S$ is a tree and $S\subseteq T$. The set $\{t\in \mathcal{AR} : st(T)\sqsubseteq t\ \&\ t\in L_{1}\}$ is infinite since $\alpha_{st(T)}\in \leftidx{^{*}}{L_{1}}$. Thus, $st(S)=st(T)$. If $s\in S/st(S)$ then there exists $n\in \mathbb{N}$ such that $s\in L_{n}$. So $\alpha_{s} \in \leftidx{^{*}}{L_{n+1}} \subseteq \leftidx{^{*}}{S}$. Hence, $S$ is an $\vec{\alpha}$-tree. Note that for all $n\in\mathbb{N}$, $L_{n}\subseteq H$. Thus
$$S/st(S) =\bigcup_{n=0}^{\infty} L_{n} \subseteq H.$$ \end{proof}

\begin{thm}\label{abstract alpha Ramsey theorem}
 Assume that $(\mathcal{R},\le, r)$ satisfies \emph{A.1}, \emph{A.2} and \emph{A.4} and for all $s\in\mathcal{AR}$, $\leftidx{^{*}}{s}=s$. For all $\mathcal{X}\subseteq \mathcal{R}$ and for all $\vec{\alpha}$-trees $T$ there exists an $\vec{\alpha}$-tree $S\subseteq T$ with $st(S)=st(T)$ such that one of the following holds:
\begin{enumerate}
\item $[S]\subseteq \mathcal{X}$.
\item $[S]\cap \mathcal{X}=\emptyset$.
\item For all $\vec{\alpha}$-trees $S'$, if $S'\subseteq S$ then $[S']\not \subseteq \mathcal{X}$ and $[S']\cap\mathcal{X} \not = \emptyset$.
\end{enumerate}
\end{thm}
\begin{proof}
Suppose that $\mathcal{X}\subseteq\mathcal{R}$ and $T$ is an $\vec{\alpha}$-tree. Consider the following sets,
$$G= \{ s \in \mathcal{AR}: \exists \mbox{$\vec{\alpha}$-tree $S\subseteq T$ with stem $s$ such that $[S]\subseteq \mathcal{X}$}\}$$ 
$$F= \{ s \in \mathcal{AR}: \exists \mbox{$\vec{\alpha}$-tree $S\subseteq T$ with stem $s$ such that $[S]\subseteq \mathcal{R}\setminus\mathcal{X}$}\}$$
$$H= \{s\in \mathcal{AR} : \forall \mbox{$\vec{\alpha}$-tree $S\subseteq T$ with stem $s$, $[S]\not\subseteq \mathcal{X}$ and $[S]\cap \mathcal{X} \not = \emptyset$}\}.$$
Notice that $H= \mathcal{AR}\setminus(G \cup F)$.

\begin{claim}
If $\alpha_{s} \not\in \leftidx{^{*}}{H}$ then $s\not \in H$.
\end{claim}
\begin{proof}
Suppose that $s\in \mathcal{AR}$ and $\alpha_{s} \not \in \leftidx{^{*}}{H} = \leftidx{^{*}}{[\mathbb{N}]^{<\infty}\setminus(\leftidx{^{*}}{G} \cup \leftidx{^{*}}{F})}$. Hence, $\alpha_{s}\in \leftidx{^{*}}{G}$ or $\alpha_{s}\in \leftidx{^{*}}{F}$.

Consider the case when $\alpha_{s}\in \leftidx{^{*}}{G}$. For each $t \in \mathcal{AR}_{|s|+1}$ such that $s\sqsubseteq t\in G$, let $T_{t}$ be an $\vec{\alpha}$-tree with stem $t$ such that $[T_{t}]\subseteq \mathcal{X}$. Let $A= \{t \in \mathcal{AR}_{|s|+1}:s\sqsubseteq t\in G\}$ and note that $\alpha_{s} \in \leftidx{^{*}}{A}$. Let $S=\bigcup_{t\in A} T_{t}$. It is clear that $S$ is a tree with stem $s$ and $A =\bigcup_{n\in A}\{st(T_{n})\} \subseteq S$. In addition, $[S] = \bigcup_{t\in A} [T_{t}] \subseteq \mathcal{X}$. If $t\in S$ then either $t=s$ or there exists $t'\in A$ such that $t\in T_{t'}/t'$. If $t=s$ then $\alpha_{t} = \alpha_{s} \in \leftidx{^{*}}{ A} =\leftidx{^{*}}{\bigcup_{u\in A}\{st(T_{u})\}} \subseteq \leftidx{^{*}}{S}.$ If there exists $t'\in A$ such that $t\in T_{n}/(t')$, then $\alpha_{t} \in \leftidx{^{*}}{T_{t'}} \subseteq \leftidx{^{*}}{S}$. So $S$ is an $\vec{\alpha}$-tree with stem $s$ such that $[S]\subseteq \mathcal{X}$. Thus, $s\in G$. In particular, $s\not \in H$.

By an identical argument, if $\alpha_{s}\in \leftidx{^{*}}{F}$ then there is an $\vec{\alpha}$-tree $S$ with stem $s$ such that $[S]\subseteq \mathcal{R} \setminus\mathcal{X}.$ In this case, we also have $s \not \in H$ as $s\in F$.\end{proof}

If $st(T)\in G$ then (1) holds. If $st(T)\in F$ then (2) holds. Otherwise $st(T)\in H$. By Lemma \ref{alpha diag} there is an $\vec{\alpha}$-tree $S\subseteq T$ such that $st(S)=st(T)$ and $S/st(S)\subseteq H$. If $S'\subseteq S$ is an $\vec{\alpha}$-tree then $st(S')\in S/st(S)\subseteq H$. Since $S'\subseteq T$, $[S']\not \subseteq \mathcal{X}$ and $[S']\cap \mathcal{X} \not=\emptyset$. So if (1) and (2) fail there is an $\vec{\alpha}$-tree showing that (3) holds.
\end{proof}
The next result is an abstract version of the $\vec{\alpha}$-Ramsey Theorem from the previous section.

\begin{cor}[Abstract $\vec{\alpha}$-Ramsey Theorem] 
 Assume that $n\in\mathbb{N}$, $(\mathcal{R},\le, r)$ satisfies \emph{A.1}, \emph{A.2} and \emph{A.4} and for all $s\in\mathcal{AR}$, $\leftidx{^{*}}{s}=s$. For all $A\subseteq \mathcal{AR}_{n}$ and for all $\vec{\alpha}$-trees $T$ there exists an $\vec{\alpha}$-tree $S\subseteq T$ with $st(S)=st(T)$ such that either $S(n)\subseteq A$ or $S(n)\cap A =\emptyset$.
\end{cor}
\begin{proof}
Let $n\in \mathbb{N}$, $A\subseteq\mathcal{AR}_{n}$ and $T$ be an $\vec{\alpha}$-tree. Let $\mathcal{X}=\{ Y\in\mathcal{R} : r_{n}(Y)\in A\}$. Notice that $\mathcal{X}$ can not satisfy conclusion (3) in the statement of Theorem \ref{abstract alpha Ramsey theorem} because any $\vec{\alpha}$-tree $S$ with $|st(S)|\ge n$ will either have $[S]\subseteq \mathcal{X}$ or $[S]\cap \mathcal{X} = \emptyset$ depending on whether $r_{n}(st(S))\in A$ or $r_{n}(st(S))\not\in A$, respectively. So Theorem \ref{alpha Ramsey theorem} implies that there is an $\vec{\alpha}$-tree $S\subseteq T$ with $st(S)=st(T)$ such that either $[S]\subseteq\mathcal{X}$ or $[S]\cap \mathcal{X}=\emptyset.$ Thus, either $S(n)\subseteq A$ or $S(n)\cap A=\emptyset$ depending on whether $[S]\subseteq\mathcal{X}$ or $[S]\cap \mathcal{X}=\emptyset,$ respectively.
\end{proof}

\subsection{The Abstract $\vec{\alpha}$-Ellentuck Theorem}
For some subsets of $\mathcal{R}$ conclusion (1) and (3) of the previous Theorem are impossible. For example, every one-element subset of $\mathcal{R}$ has the property that conclusion (1) and (3) are impossible. On the other hand, for some subsets (1) and (2) are possible but (3) is impossible. For example, for all $\vec{\alpha}$-trees $T$, (1) and (2) are possible for $[T]$ but (3) is impossible for $[T]$.

\begin{defn}
$\mathcal{X}\subseteq \mathcal{R}$ is said to be \emph{$\vec{\alpha}$-Ramsey} if for all $\vec{\alpha}$-trees $T$ there exists an $\vec{\alpha}$-tree $S\subseteq T$ with $st(S)=st(T)$ such that either $[S]\subseteq \mathcal{X}$ or $[S]\cap \mathcal{X}=\emptyset$. $\mathcal{X}$ is said to be \emph{$\vec{\alpha}$-Ramsey null} if for all $\vec{\alpha}$-trees $T$ there exists an $\vec{\alpha}$-tree $S\subseteq T$ with $st(S)=st(T)$ such that $[S]\cap \mathcal{X}=\emptyset$.
\end{defn}

\begin{cor}\label{abstract sigma ideal}
 Assume that $(\mathcal{R},\le, r)$ satisfies \emph{A.1}, \emph{A.2} and \emph{A.4} and for all $s\in\mathcal{AR}$, $\leftidx{^{*}}{s}=s$. Then the collection of $\vec{\alpha}$-Ramsey null sets is a $\sigma$-ideal.
\end{cor}
\begin{proof}
It is clear that $\emptyset$ is $\vec{\alpha}$-Ramsey null. It should also be clear from the definition that if $\mathcal{X}$ is $\vec{\alpha}$-Ramsey null and $\mathcal{Y}\subseteq \mathcal{X}$ then $\mathcal{Y}$ is also $\vec{\alpha}$-Ramsey null. So it is enough to show that the countable union of $\vec{\alpha}$-Ramsey null sets is $\vec{\alpha}$-Ramsey null. To this end, let $\left<\mathcal{X}_{i}:i\in\mathbb{N}\right>$ be a sequence of $\vec{\alpha}$-Ramsey null sets. 

Let $T$ be an $\vec{\alpha}$-tree. Let $S_{st(T)}\subseteq T$ be an $\vec{\alpha}$-tree such that $[S_{st(T)}]\cap \mathcal{X}_{0} =\emptyset$, such a tree exists as $\mathcal{X}_{0}$ is $\vec{\alpha}$-Ramsey null. Suppose that $S_{t}$ has been defined, is an $\vec{\alpha}$-tree, $st(S_{t})=t$ and $[S_{t}]\cap\mathcal{X}_{|t|-|st(T)|} =\emptyset$; note that $|t|-|st(T)|$ is a non-negative integer. For each $s\in S_{t}/t$ such that $|s|=|t|+1$, since $\mathcal{X}_{|s|-|st(T)|}$ is $\vec{\alpha}$-Ramsey null, there exists an $\vec{\alpha}$-tree $S_{s}\subseteq S_{t}$ such that $st(S_{s})=s$ and $[S_{s}]\cap \mathcal{X}_{|s|-|st(T)|}=\emptyset$. Then let
$$\begin{cases}
L_{0} = \{ st(T)\},\\
L_{n+1} =\{t\in \mathcal{AR} : \exists s\in L_{n}, \ s\sqsubseteq t, \ |t|=|s|+1 \ \& \ t\in S_{s}\}. 
\end{cases}$$
Let 
$$ S = \{ t\in\mathcal{AR} : t\sqsubseteq st(T) \} \cup \bigcup_{n=0}^{\infty} L_{n}.$$
It is clear that $S$ is a tree with $st(S)=st(T)$. If $s\in S/st(S)$ then there exists $n\in\mathbb{N}$ such that $s\in L_{n}$. Since $S_{s}$ is an $\vec{\alpha}$-tree $\alpha_{s}\in \leftidx{^{*}}{S_{s}}$. Hence, $\alpha_{s} \in \leftidx{^{*}}{L_{n+1}}\subseteq \leftidx{^{*}}{S}$. So $S$ is an $\vec{\alpha}$-tree. 

If $X\in[S]$ and $n\in \mathbb{N}$ then there exists $t\sqsubseteq X$ such that $|t|=|st(T)|+n$. If $t \sqsubset t' \sqsubseteq X$ then $t'\in S_{r_{|t'|-1}(t')} \subseteq S_{t}$. Since $S_{t}$ is an $\vec{\alpha}$-tree with $st(S_{t})=t$, $X\in [S_{t}]$. Note that $[S_{t}]\cap \mathcal{X}_{n} =[S_{t}]\cap \mathcal{X}_{|t|-|st(T)|} = \emptyset$. Hence, $X\not \in \mathcal{X}_{n}$. Since $X$ was an arbitrary element of $[S]$ and $n$ an arbitrary element of $\mathbb{N}$,
$$[S]\cap \bigcup_{n=0}^{\infty} \mathcal{X}_{n} = \emptyset.$$
Thus, $\bigcup_{n=0}^{\infty} \mathcal{X}_{n}$ is $\vec{\alpha}$-Ramsey null.
\end{proof}
\begin{cor}  If $(\mathcal{R},\le, r)$ satisfies \emph{A.1}, \emph{A.2} and \emph{A.4} and for all $s\in\mathcal{AR}$, $\leftidx{^{*}}{s}=s$ then $\mathcal{R}$ is uncountable. 
\end{cor}
\begin{proof}
Let $X\in \mathcal{R}$. By Lemma \ref{infinity lemma}, for all $\vec{\alpha}$-trees $T$, $[T]$ is infinite. Thus for all $\vec{\alpha}$-trees $T$, $[T]\not\subseteq \{X\}$. Therefore conclusion (1) in Theorem \ref{abstract alpha Ramsey theorem} is not possible for $\{X\}$. Conclusion (3) is also not possible as it implies that there is an $\vec{\alpha}$-tree $S$ such that for all $\vec{\alpha}$-trees $S'\subseteq S$, $[S']\cap\{X\}\not=\emptyset$, a contradiction. So, Theorem \ref{abstract alpha Ramsey theorem} implies that $\{X\}$ is $\vec{\alpha}$-Ramsey null.

Toward a contradiction suppose that $\mathcal{R}$ is countable. By Corollary \ref{abstract sigma ideal}, $\mathcal{R}$ would be the countable union of $\vec{\alpha}$-Ramsey null sets and hence $\vec{\alpha}$-Ramsey null itself. This is a contradiction since for all $\vec{\alpha}$-trees $S$, $[S]\subseteq \mathcal{R}$. 
\end{proof}
The next Lemma, among other things, shows that for all $\vec{\alpha}$-trees $T$, $[T]$ is $\vec{\alpha}$-Ramsey. 

\begin{lem}\label{abstract intersection alpha trees} 
 Assume that $(\mathcal{R},\le, r)$ satisfies \emph{A.1}, \emph{A.2} and \emph{A.4} and for all $s\in\mathcal{AR}$, $\leftidx{^{*}}{s}=s$. If $T$ and $S$ are $\vec{\alpha}$-trees then $S\cap T$ is an $\vec{\alpha}$-tree if and only if either $st(T)\sqsubseteq st(S)\in T$ or $st(S)\sqsubseteq st(T)\in S$.
\end{lem}
\begin{proof} Let $S$ and $T$ be $\vec{\alpha}$-trees. If $st(T)$ and $st(S)$ are not $\sqsubseteq$-comparable then $S\cap T =\{t\in\mathcal{AR} : t\sqsubseteq st(S) \ \& \ t\sqsubseteq st(T)\}$ and can not be an $\vec{\alpha}$-tree as it is a finite set (Note that by Lemma \ref{infinity lemma} every $\vec{\alpha}$-tree must be infinite). By contrapostive, if $S\cap T$ is an $\vec{\alpha}$-tree then either $st(S)\sqsubseteq st(T)$ or $st(T)\sqsubseteq st(S)$.  If $st(S)\sqsubseteq st(T)$ and $st(T)\not\in S$ then $S\cap T\subseteq \{t\in\mathcal{AR} : t\sqsubseteq st(T)\}$ can not be an $\vec{\alpha}$-tree as it is a finite set. Hence, if $S\cap T$ is an $\vec{\alpha}$-tree and $st(S)\sqsubseteq st(T)$ then $st(T) \in S$. Likewise, if $S\cap T$ is an $\vec{\alpha}$-tree and $st(T)\sqsubseteq st(S)$ then $st(S)\in T$. Altogether we have shown that if $S\cap T$ is an $\vec{\alpha}$-tree then either $st(S)\sqsubseteq st(T)\in S$ or $st(T)\sqsubseteq st(s)\in T$.

 Let $S$ and $T$ be $\vec{\alpha}$-trees and suppose $st(S)\sqsubseteq st(T)\in S$. Then $ (S\cap T)/st(S \cap T) =(S\cap T)/st(T) = S/st(T) \cap T/st(T) = S/st(S) \cap T/st(T)\not=\emptyset$.  Since $S$ and $T$ are both $\vec{\alpha}$-trees, for each $s\in (S\cap T)/st(S\cap T)$, $\alpha_{s} \in \leftidx{^{*}}{S} \cap \leftidx{^{*}}{T}=\leftidx{^{*}}{(S\cap T)}$. That is, $S\cap T$ is an $\vec{\alpha}$-tree. By symmetry, if either $st(T)\sqsubseteq st(S)\in T$ or $st(S)\sqsubseteq st(T)\in S$ then $S\cap T$ is an $\vec{\alpha}$-tree.
\end{proof}

The $\vec{\alpha}$-Ramsey and $\vec{\alpha}$-Ramsey null sets can be completely characterized in terms of topological notions with respect to a space on $\mathcal{R}$ generated from the $\vec{\alpha}$-trees. 

\begin{lem}
 Assume that $(\mathcal{R},\le, r)$ satisfies \emph{A.1}, \emph{A.2} and \emph{A.4} and for all $s\in\mathcal{AR}$, $\leftidx{^{*}}{s}=s$.
The collection $\{[T]: \mbox{$T$ is an $\vec{\alpha}$-tree}\}$ is a basis for a topology on $\mathcal{R}$.
\end{lem}
\begin{proof}
It is enough to show that, if $S$ and $T$ are $\vec{\alpha}$-trees and $[S]\cap[T]\not=\emptyset$ then $S\cap T$ is an $\vec{\alpha}$-tree.
Suppose $S$ and $T$ are $\vec{\alpha}$-trees and $X\in [S]\cap[T]$. Then either $st(T)\sqsubseteq st(S) \sqsubseteq X$ or $st(S)\sqsubseteq st(T) \sqsubseteq X$. If $st(S)\sqsubseteq X$ then $st(S) \in T$ since $X\in[T]$. Likewise, if $st(T)\sqsubseteq X$ then $st(T) \in S$. So, either $st(S)\sqsubseteq st(T)\in S$ or $st(T)\sqsubseteq st(S)\in T$. By Lemma \ref{abstract intersection alpha trees}, $S\cap T$ is an $\vec{\alpha}$-tree.
\end{proof}

\begin{defn} 
 Assume that $(\mathcal{R},\le, r)$ satisfies {A.1}, {A.2} and {A.4} and for all $s\in\mathcal{AR}$, $\leftidx{^{*}}{s}=s$.
The topology on $\mathcal{R}$ generated by $\{[T] : \mbox{$T$ is an $\vec{\alpha}$-tree}\}$ is called \emph{the $\vec{\alpha}$-Ellentuck topology}. We say that a subset of $\mathcal{R}$ is \emph{$\vec{\alpha}$-open} if it is open with respect the $\vec{\alpha}$-Ellentuck topology.
\end{defn}

\begin{rem}
We leave it to the interested reader to show that, if $(\mathcal{R},\le, r)$ satisfies {A.1}, {A.2} and {A.4} and for all $s\in\mathcal{AR}$, $\leftidx{^{*}}{s}=s$, then the \emph{$\vec{\alpha}$-Ellentuck space} is a zero-dimensional Baire space on $\mathcal{R}$. 
\end{rem}

\begin{cor}\label{abstract open is ramsey} 
 Assume that $(\mathcal{R},\le, r)$ satisfies \emph{A.1}, \emph{A.2} and \emph{A.4} and for all $s\in\mathcal{AR}$, $\leftidx{^{*}}{s}=s$.
Every $\vec{\alpha}$-open set is $\vec{\alpha}$-Ramsey.
\end{cor}
\begin{proof} We show that, if $\mathcal{X}\subseteq\mathcal{R}$ is not $\vec{\alpha}$-Ramsey then there exists $X\in \mathcal{X}$ such that for all $\vec{\alpha}$-trees $S$, if $X\in [S]$ then  $[S]\not\subseteq \mathcal{X}$. In other words, if $\mathcal{X}$ is not $\vec{\alpha}$-Ramsey then it contains a point not in its interior with respect to the $\vec{\alpha}$-Ellentuck topology. The result follows by taking the contrapositve of this statement.

Suppose that $\mathcal{X}$ is not $\vec{\alpha}$-Ramsey. Then there exists an $\vec{\alpha}$-tree $T$ such that for each $\vec{\alpha}$-tree $S\subseteq T$ with $st(S)=st(T)$, $[S]\not\subseteq \mathcal{X}$ and $[S]\cap \mathcal{X}\not = \emptyset$. By Theorem \ref{abstract alpha Ramsey theorem} there is an $\vec{\alpha}$-tree $S\subseteq T$ with $st(S)=st(T)$ such that for all $\vec{\alpha}$-trees $S'\subseteq S$, $[S']\not \subseteq \mathcal{X}$ and $[S']\cap \mathcal{X}\not=\emptyset$. 

Since $S\subseteq S$, $[S]\cap \mathcal{X} \not =\emptyset$. Let $X$ be any element of $[S]\cap \mathcal{X}$. If $S'$ is an $\vec{\alpha}$-tree and $X\in[S']$ then $X\in[S]\cap [S']$. So either $st(S')\sqsubseteq st(S) \sqsubseteq X$ or $st(S)\sqsubseteq st(S') \sqsubseteq X$. If $st(S)\sqsubseteq X$ then $st(S) \in S'$ since $X\in[S']$. Likewise, if $st(S')\sqsubseteq X$ then $st(S') \in S$. By Lemma \ref{abstract intersection alpha trees}, $S'\cap S$ is an $\vec{\alpha}$-tree. Since $S'\cap S\subseteq S$, the previous paragraph shows that $[S'\cap S]\not \subseteq \mathcal{X}$ and $[S'\cap S] \cap \mathcal{X} \not = \emptyset$. In particular, $[S']\not \subseteq \mathcal{X}$ as $[S'\cap S]\subseteq [S']$.
\end{proof}

\begin{defn}
$\mathcal{X}\subseteq \mathcal{R}$ \emph{is $\vec{\alpha}$-nowhere dense/ is $\vec{\alpha}$-meager/ has the $\vec{\alpha}$-Baire property} if it is nowhere dense/ is meager/ has the Baire property with respect to the $\vec{\alpha}$-Ellentuck topology. 

We say that $(\mathcal{R},\vec{\alpha},\le,r)$ is an \emph{$\vec{\alpha}$-Ramsey space} if the collection of $\vec{\alpha}$-Ramsey sets coincides with the $\sigma$-algebra of sets with the $\vec{\alpha}$-Baire property and the collection of $\vec{\alpha}$-Ramsey null sets coincides with the $\sigma$-ideal of $\vec{\alpha}$-meager sets.
\end{defn}

\begin{thm}[The Abstract $\vec{\alpha}$-Ellentuck Theorem] \label{abstract alpha-Ellentuck theorem}
 If $(\mathcal{R},\le, r)$ satisfies \emph{A.1}, \emph{A.2} and \emph{A.4} and for all $s\in\mathcal{AR}$, $\leftidx{^{*}}{s}=s$ then $(\mathcal{R},\vec{\alpha},\le,r)$ is an $\vec{\alpha}$-Ramsey space.

\end{thm}
\begin{proof} First note that it is clear from the definitions that every $\vec{\alpha}$-Ramsey null set is $\vec{\alpha}$-nowhere dense. Let $\mathcal{X}\subseteq \mathcal{R}$ be $\vec{\alpha}$-nowhere dense and $T$ be an $\vec{\alpha}$-tree. Note that conclusion (1) in Theorem \ref{abstract alpha Ramsey theorem} is not possible for $\mathcal{X}$ because otherwise it would not be $\vec{\alpha}$-nowhere dense. Similarly, conclusion (3) in Theorem \ref{abstract alpha Ramsey theorem} is not possible for $\mathcal{X}$ because otherwise it would not be $\vec{\alpha}$-nowhere dense. So, by Theorem \ref{abstract alpha Ramsey theorem}, there exists an $\vec{\alpha}$-tree $S\subseteq T$ with $st(S)=s(T)$ such that $[S]\cap\mathcal{X}=\emptyset$. Hence, $\mathcal{X}$ is $\vec{\alpha}$-Ramsey null. 

If $\mathcal{X}$ is $\vec{\alpha}$-meager then $\mathcal{X}$ is the countable union of $\vec{\alpha}$-nowhere dense sets. The previous paragraph and Corollary \ref{abstract sigma ideal} imply that $\mathcal{X}$ is $\vec{\alpha}$-Ramsey null. On the other hand, if $\mathcal{X}$ is $\vec{\alpha}$-Ramsey null then it is also $\vec{\alpha}$-meager since the last paragraph implies that it is $\vec{\alpha}$-nowhere dense. Therefore, the collection of $\vec{\alpha}$-meager sets coincides with the $\sigma$-ideal of $\vec{\alpha}$-Ramsey null sets. 

Suppose that $\mathcal{X}$ has the $\vec{\alpha}$-Baire property. Then there is an $\vec{\alpha}$-open set $\mathcal{O}$ and an $\vec{\alpha}$-meager set $\mathcal{M}$ such that $\mathcal{X} = \mathcal{O} \Delta \mathcal{M}$. Corollary \ref{abstract open is ramsey} and the previous paragraph imply that $\mathcal{O}$ and $\mathcal{M}$ are $\vec{\alpha}$-Ramsey and $\vec{\alpha}$-Ramsey null, respectively. If $T$ is an $\vec{\alpha}$-tree then there exists $\vec{\alpha}$-trees $S'\subseteq S\subseteq T$ with $st(S')=st(S)=st(T)$ such that $[S]\cap \mathcal{M} =\emptyset$ and either $[S']\subseteq \mathcal{O}$ or $[S']\cap \mathcal{O}=\emptyset$. So either, $[S'] \cap (\mathcal{O}\Delta \mathcal{M})=\emptyset$ or $[S']\subseteq \mathcal{O}\Delta \mathcal{M}$. Hence, $\mathcal{X}$ is $\vec{\alpha}$-Ramsey.

Let $\mathcal{X}$ be an $\vec{\alpha}$-Ramsey set. Let $\mathcal{O}$ be the interior of $\mathcal{X}$ with respect to the $\vec{\alpha}$-Ellentuck topology.  $\mathcal{O}$ is $\vec{\alpha}$-Ramsey by Corollary \ref{abstract open is ramsey}. So, $\mathcal{X} \setminus \mathcal{O}$ is $\vec{\alpha}$-Ramsey by Corollary \ref{abstract sigma ideal}. So, for all $\vec{\alpha}$-trees $T$ there exists an $\vec{\alpha}$-tree $S\subseteq T$ such that either $[S]\subseteq \mathcal{X}\setminus \mathcal{O}$ or $[S]\cap (\mathcal{X}\setminus\mathcal{O})=\emptyset$. $[S]\subseteq \mathcal{X}\setminus \mathcal{O}$ is not possible because it would mean that $[S]\subseteq\mathcal{O}$ as it is an $\vec{\alpha}$-open set contained in $\mathcal{X}$. Thus, $[S]\cap (\mathcal{X}\setminus\mathcal{O})=\emptyset$. So $\mathcal{X}\setminus\mathcal{O}$ is $\vec{\alpha}$-Ramsey null. By the second paragraph of this proof, $\mathcal{X}\setminus \mathcal{O}$ is $\vec{\alpha}$-meager. $\mathcal{X}$ has the $\vec{\alpha}$-Baire property since $\mathcal{X} = \mathcal{O} \Delta (\mathcal{X}\setminus \mathcal{O})$. 

The previous two paragraphs show that the collection of sets with the $\vec{\alpha}$-Baire property coincides with the $\sigma$-algebra of $\vec{\alpha}$-Ramsey sets. 
\end{proof}
\subsection{Abstract Ultra-Ramsey Theory} In this subsection we extend ultra-Ramsey Theory to the abstract setting using the abstract $\vec{\alpha}$-Ellentuck Theorem. Recall that, an \emph{ultrafilter $\mathcal{U}$ on $X$} is a subset of $\wp(X)$ such that for all subsets $A$ and $B$ of $X$,
\begin{multicols}{2}
\begin{enumerate}
\item $\emptyset \not \in \mathcal{U} \ \& \ X\in \mathcal{U}$,
\item $A\cup B \in \mathcal{U} \Leftrightarrow A\in U \mbox{ or } B\in U$,
\item $A \cap B \in \mathcal{U} \Leftrightarrow A \in U \ \& \ B\in U$,
\item $A\in \mathcal{U} \Leftrightarrow X\setminus A \not \in \mathcal{U}$.
\end{enumerate}
\end{multicols} 

\begin{prop}\label{need in milliken example}
 Assume that $(\mathcal{R},\le, r)$ satisfies \emph{A.1}, \emph{A.2} and \emph{A.4} and for all $s\in\mathcal{AR}$, $\leftidx{^{*}}{s}=s$. Suppose that the $\mathfrak{c}^{+}$-enlarging property holds, $s\in\mathcal{AR}$ and $X=\{t\in\mathcal{AR}_{|s|+1} : s\sqsubseteq t\}$. $\mathcal{U}$ is an ultrafilter on $X$ if and only if there exists $\beta \in \leftidx{^{*}}{X}$ such that 
$$\mathcal{U} = \{ A\subseteq X : \beta \in \leftidx{^{*}}{A}\}.$$
 Moreover, $\mathcal{U}$ is non-principal if and only if $\beta \not \in \mathcal{AR}_{|s|+1}$.
\end{prop}
\begin{proof} If $\beta\in\leftidx{^{*}}{X}$ then Proposition \ref{star transform} (6), (7) and (8) immediately show that $\{ A\subseteq X : \beta \in \leftidx{^{*}}{A}\}$ satisfies (2), (3) and (4) in the definition of an ultrafilter. Clearly, $\beta\not \in \leftidx{^{*}}{\emptyset}=\emptyset$. By assumption, $\beta \in \leftidx{^{*}}{X}$. Thus $\{A\subseteq X : \beta \in \leftidx{^{*}}{A}\}$ is an ultrafilter on $X$.

If $\mathcal{U}$ is an ultrafilter on $X$ then $|\mathcal{U}|\le\mathfrak{c}$ and $\mathcal{U}$ has the finite intersection property. By the $\mathfrak{c}^{+}$-enlarging property, $\bigcap_{A\in \mathcal{U}} \leftidx{^{*}}{A}\not=\emptyset.$ Let $\beta$ be any element of this intersection. If $A\in \mathcal{U}$ then $ \bigcap_{B\in \mathcal{U}} \leftidx{^{*}}{B}\subseteq \leftidx{^{*}}{A}.$ So for all $A\in \mathcal{U}$,  $\beta \in \leftidx{^{*}}{A}$. If $ A\not \in \mathcal{U}$ then $X\setminus A\in \mathcal{U}$ since $(X\setminus A)\cup A = X \in \mathcal{U}$. Thus $\beta \in \leftidx{^{*}}{(X\setminus A)} = \leftidx{^{*}}{X}\setminus \leftidx{^{*}}{A}$. In other words, $\beta\not \in \leftidx{^{*}}{A}.$ Thus, $\mathcal{U} = \{ A\subseteq X : \beta \in \leftidx{^{*}}{A}\}.$

Suppose that $\mathcal{U}$ is principal. Then there exists a finite set $\{a_{0}, a_{1}, \dots , a_{n}\}\subseteq \mathcal{AR}$ such that $\beta \in\leftidx{^{*}}{\{a_{0}, a_{1}, \dots , a_{n}\}}.$ By the pair axiom $$\beta \in\leftidx{^{*}}{\{a_{0}, a_{1}, \dots , a_{n}\}}=\{\leftidx{^{*}}{a_{0}}, \leftidx{^{*}}{a_{1}}, \dots , \leftidx{^{*}}{a_{n}}\}=\{a_{0}, a_{1}, \dots , a_{n}\}.$$ So there exists $i\le n$ such that $\beta = a_{i}$. In particular, $\beta\in \mathcal{AR}$. 

If $\beta \in \mathcal{AR}$ then $\beta \in \{\beta\} = \leftidx{^{*}}{\{\beta\}}$. So $\{ A\subseteq \mathbb{N} : \beta \in \leftidx{^{*}}{A}\}$ contains a finite set and $\mathcal{U}$ is principal.
\end{proof}

\begin{defn}
Let $\vec{\mathcal{U}}=\left<\mathcal{U}_{s} : s\in \mathcal{AR} \right>$ be a sequence of non-principal ultrafilters such that for all $s\in\mathcal{AR}$, $\mathcal{U}_{s}$ is an ultrafilter on $\{t\in\mathcal{AR}_{|s|+1} : s\sqsubseteq t\}$. A tree $T$ on $\mathcal{R}$ with stem $st(T)$ is a \emph{$\vec{\mathcal{U}}$-tree} if for all $s\in T/st(T)$
$$ \{t\in\mathcal{AR}_{|s|+1}:s\sqsubseteq t \in T\} \in \mathcal{U}_{s}.$$
\end{defn}

\begin{prop}\label{abstract U-tree alpha-tree}
 Assume that $(\mathcal{R},\le, r)$ satisfies \emph{A.1}, \emph{A.2} and \emph{A.4} and for all $s\in\mathcal{AR}$, $\leftidx{^{*}}{s}=s$. Suppose that the $\mathfrak{c}^{+}$-enlarging property holds.  For all sequences $\vec{\mathcal{U}}=\left<\mathcal{U}_{s} : s\in \mathcal{AR} \right>$ of non-principal ultrafilters such that for all $s\in\mathcal{AR}$, $\mathcal{U}_{s}$ is an ultrafilter on $\{t\in\mathcal{AR}_{|s|+1} : s\sqsubseteq t\}$, there exists a sequence $\vec{\alpha}=\left<\alpha_{s} : s\in\mathcal{AR}\right>$ of elements of $\leftidx{^{*}}{\mathcal{AR}}\setminus\mathcal{AR}$ such that for all trees $T$ on $\mathcal{R}$, $T$ is an $\vec{\mathcal{U}}$-tree if and only if $T$ is an $\vec{\alpha}$-tree.
\end{prop}
\begin{proof}
The previous proposition, allows us to chose a sequence $\left<\alpha_{s} : s\in \mathcal{AR} \right>$ of elements of $\leftidx{^{*}}{\mathcal{AR}}\setminus\mathcal{AR}$ such that for all $s\in \mathcal{AR}$ and for all $A\subseteq\{t\in\mathcal{AR}_{|s|+1}:s\sqsubseteq t\}$, $A\in\mathcal{U}_{s} \iff \alpha_{s} \in \leftidx{^{*}}{A}.$ In particular, for all trees $T$ with stem $st(T)$, for all $s\in T/st(T)$,
 $$\{t\in\mathcal{AR}_{|s|+1} :s\sqsubseteq t \in T\} \in \mathcal{U}_{s} \iff \alpha_{s}\in \leftidx{^{*}}{T}.$$
\end{proof}

For the remainder of this section, we fix a sequence $\vec{\mathcal{U}}=\left<\mathcal{U}_{s} : s\in \mathcal{AR} \right>$ of non-principal ultrafilters such that for all $s\in\mathcal{AR}$, $\mathcal{U}_{s}$ is an ultrafilter on $\{t\in\mathcal{AR}_{|s|+1} : s\sqsubseteq t\}$. All definitions are taken with respect to this sequence.

\begin{thm} \label{abstract ultra-Ramsey theorem}
 Assume that $(\mathcal{R},\le, r)$ satisfies \emph{A.1}, \emph{A.2} and \emph{A.4} and for all $s\in\mathcal{AR}$, $\leftidx{^{*}}{s}=s$. If the $\mathfrak{c}^{+}$-enlarging property holds, then  for all $\mathcal{X}\subseteq \mathcal{R}$ and for all $\vec{\mathcal{U}}$-trees $T$ there exists an $\vec{\mathcal{U}}$-tree $S\subseteq T$ with $st(S)=st(T)$ such that one of the following holds:
\begin{enumerate}
\item $[S]\subseteq \mathcal{X}$.
\item $[S]\cap \mathcal{X}=\emptyset$.
\item For all $\vec{\mathcal{U}}$-trees $S'$, if $S'\subseteq S$ then $[S']\not \subseteq \mathcal{X}$ and $[S']\cap\mathcal{X} \not = \emptyset$.
\end{enumerate}
\end{thm}
\begin{proof}
By Proposition \ref{abstract U-tree alpha-tree} there exists a sequence $\vec{\alpha}=\left<\alpha_{s} : \mathcal{AR}\right>$ such that for all trees $T$ on $\mathcal{R}$, $T$ is an $\vec{\mathcal{U}}$-tree if and only if $T$ is an $\vec{\alpha}$-tree. This result is simply a restatement of Theorem \ref{abstract alpha Ramsey theorem} using $\vec{\mathcal{U}}$-trees instead of $\vec{\alpha}$-trees.
\end{proof}

\begin{defn} $\mathcal{X}\subseteq \mathcal{R}$ is said to be \emph{$\vec{\mathcal{U}}$-Ramsey} if for all $\vec{\mathcal{U}}$-trees $T$ there exists an $\vec{\mathcal{U}}$-tree $S\subseteq T$ with $st(S)=st(T)$ such that either $[S]\subseteq \mathcal{X}$ or $[S]\cap \mathcal{X}=\emptyset$. $\mathcal{X}$ is said to be \emph{$\vec{\mathcal{U}}$-Ramsey null} if for all $\vec{\mathcal{U}}$-trees $T$ there exists an $\vec{\mathcal{U}}$-tree $S\subseteq T$ with $st(S)=st(T)$ such that $[S]\cap \mathcal{X}=\emptyset$.

The \emph{$\vec{\mathcal{U}}$-Ellentuck space} is the topological space on $\mathcal{R}$ generated by $\{[T]:T$ is a $\vec{\mathcal{U}}$-tree$\}$. $\mathcal{X}\subseteq [\mathbb{N}]^{\infty}$ \emph{is $\vec{\mathcal{U}}$-nowhere dense/ is $\vec{\mathcal{U}}$-meager/ has the $\vec{\mathcal{U}}$-Baire property} if it is nowhere dense/ is meager/ has the Baire property with respect to the $\vec{\mathcal{U}}$-Ellentuck space.

We say that $(\mathcal{R},\vec{\mathcal{U}},\le,r)$ is an \emph{ultra-Ramsey space} if the collection of $\vec{\mathcal{U}}$-Ramsey sets coincides with the $\sigma$-algebra of sets with the $\vec{\mathcal{U}}$-Baire property and the collection of $\vec{\mathcal{U}}$-Ramsey null sets coincides with the $\sigma$-ideal of $\vec{\mathcal{U}}$-meager sets.
\end{defn}

\begin{thm}[The Abstract Ultra-Ellentuck Theorem] \label{abstract ultra-Ellentuck theorem}
 Assume that $(\mathcal{R},\le, r)$ satisfies \emph{A.1}, \emph{A.2} and \emph{A.4} and for all $s\in\mathcal{AR}$, $\leftidx{^{*}}{s}=s$. If the $\mathfrak{c}^{+}$-enlarging property holds, then $(\mathcal{R},\vec{\mathcal{U}},\le,r)$ is an ultra-Ramsey space.
\end{thm}
\begin{proof}
By Proposition \ref{U-tree alpha-tree} there exists a sequence $\vec{\alpha}=\left<\alpha_{s} :s\in \mathcal{AR}\right>$ such that for all trees $T$ on $\mathcal{R}$, $T$ is an $\vec{\mathcal{U}}$-tree if and only if $T$ is an $\vec{\alpha}$-tree. This result is simply a restatement of the Abstract $\vec{\alpha}$-Ellentuck Theorem in the setting of $\vec{\mathcal{U}}$-trees instead of $\vec{\alpha}$-trees.
\end{proof}

\subsection{An application to abstract local Ramsey theory} In this section we generalize Theorem \ref{selective ultrafilter theorem} to the abstract setting of triples $(\mathcal{R},\le,r)$. We then apply the theorem to some concrete examples of triples that form topological Ramsey spaces.
\begin{lem} Assume that $(\mathcal{R},\le, r)$ satisfies \emph{A.1}, \emph{A.2} and \emph{A.4} and for all $s\in\mathcal{AR}$, $\leftidx{^{*}}{s}=s$. Let $$\mathcal{R}_{\vec{\alpha}}= \{ X\in\mathcal{R} : \forall s\in \mathcal{AR}\restriction X, \ \alpha_{s}\in \leftidx{^{*}}{r_{|s|+1}[s,X]}\}.$$ If $X\in \mathcal{R}_{\vec{\alpha}}$ and $s\sqsubseteq X$ then there exists an $\vec{\alpha}$-tree $T$ such that $st(T)=s$ and $[T]=[s,X]$.
\end{lem}
\begin{proof}
We build $T$ level-by-level, recursively. 
$$\begin{cases}
L_{0}= \{s\} &\\
L_{n+1} = \{ t\in \mathcal{AR} : \exists t'\in L_{n}, \ t\in r_{|t'|+1}[t',X]\}.&
\end{cases}$$
Let $$T = \{t\in\mathcal{AR} : t\sqsubseteq s\}\cup \bigcup_{n=0}^{\infty} L_{n}.$$
It is clear that $T$ is a tree on $\mathcal{R}$, $st(T)=s$ and $[T]=[s,X]$. If $u\in T/st(T)$ then there exists $n\in\mathbb{N}$ such that $u\in L_{n}$. Since $X\in \mathcal{R}_{\vec{\alpha}}$, $\alpha_{u}\in \leftidx{^{*}}{r_{|u|+1}[s,X]}$. Thus $\alpha_{u}\in \leftidx{^{*}}{L_{n+1}} \subseteq \leftidx{^{*}}{T}.$ In particular, $T$ is an $\vec{\alpha}$-tree.
\end{proof}

\begin{thm}
 Assume that $(\mathcal{R},\le, r)$ satisfies \emph{A.1}, \emph{A.2} and \emph{A.4} and for all $s\in\mathcal{AR}$, $\leftidx{^{*}}{s}=s$. Let 
$$\mathcal{R}_{\vec{\alpha}}=\{ X\in \mathcal{R} : \forall s\in \mathcal{AR}\restriction X, \ \alpha_{s}\in\leftidx{^{*}}{r_{|s|+1}[s,X]}\}.$$ If for all $\vec{\alpha}$-trees $T$ there exists $X\in \mathcal{R}_{\vec{\alpha}}$ such that $\emptyset\not=[st(T),X]\subseteq [T]$, then $(\mathcal{R}, \mathcal{R}_{\vec{\alpha}}, \le, r)$ is a topological Ramsey space.
\end{thm}
\begin{proof} 
Note that by the Abstract $\vec{\alpha}$-Ellentuck Theorem $(\mathcal{R},\vec{\alpha}, \le ,r)$ is an $\vec{\alpha}$-Ramsey space.
\begin{claim}
Let $\mathcal{X}\subseteq \mathcal{R}$. $\mathcal{X}$ has the $\vec{\alpha}$-Baire property ( is $\vec{\alpha}$-meager) if and only if $\mathcal{X}$ has the $\mathcal{R}_{\vec{\alpha}}$-Baire property (is $\mathcal{R}_{\vec{\alpha}}$-meager).
\end{claim}
\begin{proof}
 Note that, if $T$ is a $\vec{\alpha}$-tree then for all $s\in T/st(T)$ there exists $X_{s}\in\mathcal{R}_{\vec{\alpha}}$ such that $s\sqsubseteq X_{s}$ and $[s,X_{s}] \subseteq [T].$ Thus for each $\vec{\alpha}$-tree $T$, 
$$[T]=\bigcup_{s\in T/st(T)} [s, X_{s}].$$
 In particular, $\{[s,X] : X\in\mathcal{R}_{\vec{\alpha}}\}$ generates the $\vec{\alpha}$-Ellentuck topology. Since the $\vec{\alpha}$-Ellentuck topology and the $\mathcal{R}_{\vec{\alpha}}$-Ellentuck topology are identical, the notions of the $\vec{\alpha}$-Baire property and the $\mathcal{R}_{\vec{\alpha}}$-Baire property coincide. Likewise, the notions of $\vec{\alpha}$-meager and $\mathcal{R}_{\vec{\alpha}}$-meager coincide.
\end{proof}
\begin{claim}
Let $\mathcal{X}\subseteq \mathcal{R}$. $\mathcal{X}$ is $\vec{\alpha}$-Ramsey ($\vec{\alpha}$-Ramsey null) if and only if $\mathcal{X}$ is $\mathcal{R}_{\vec{\alpha}}$-Ramsey ($\mathcal{R}_{\vec{\alpha}}$-Ramsey null).
\end{claim}
\begin{proof}
($\implies$) Suppose that $\mathcal{X}$ is $\vec{\alpha}$-Ramsey, $X\in\mathcal{R}_{\vec{\alpha}}$ and $[s,X]\not=\emptyset$. By the previous lemma there exists an $\vec{\alpha}$-tree $T$ with $st(T)=s$ such that $[T]=[s,X]$. So there exists an $\vec{\alpha}$-tree $S\subseteq T$ such that $st(S)=st(T)=s$ and either $[S]\subseteq \mathcal{X}$ or $[S]\cap \mathcal{X}=\emptyset$. By hypothesis, there exists $Y\in\mathcal{R}_{\vec{\alpha}}$ such that $\emptyset\not=[s,Y]\subseteq [S]$. Hence, $Y\in[s,X]\cap\mathcal{R}_{\vec{\alpha}}$ and either $[s,Y]\subseteq \mathcal{X}$ or $[s,Y]\cap \mathcal{X}=\emptyset$. In other words, $\mathcal{X}$ is $\mathcal{R}_{\vec{\alpha}}$-Ramsey.

($\impliedby$) Suppose that $\mathcal{X}$ is $\mathcal{R}_{\vec{\alpha}}$-Ramsey and $T$ be an $\vec{\alpha}$-tree. By hypothesis, there exists $X\in \mathcal{R}_{\vec{\alpha}}$ such that $\emptyset\not=[st(T),X]\subseteq [T]$. Hence there exists $Y\in[st(T),X]\cap\mathcal{R}_{\vec{\alpha}}$ such that either $[st(T),Y]\subseteq \mathcal{X}$ or $[st(T),Y]\cap\mathcal{X}=\emptyset$. By the previous lemma there exists an $\vec{\alpha}$-tree $S$ such that $st(S)=st(T)$ and $[S]\subseteq [s,Y]$. Hence, $S\subseteq T$, $st(S)=st(T)$ and either $[S]\subseteq \mathcal{X}$ or $[S]\cap \mathcal{X}=\emptyset$. That is, $\mathcal{X}$ is $\vec{\alpha}$-Ramsey.
\end{proof}

Let $\mathcal{X}\subseteq \mathcal{R}$. The previous two claims and the fact that $(\mathcal{R},\vec{\alpha}, \le ,r)$ is an $\vec{\alpha}$-Ramsey space imply that,
$\mathcal{X}$ is $\mathcal{R}_{\vec{\alpha}}$-Ramsey (is $\mathcal{R}_{\vec{\alpha}}$-Ramsey null) iff $\mathcal{X}$ is $\vec{\alpha}$-Ramsey (is $\vec{\alpha}$-Ramsey null) iff $\mathcal{X}$ has the $\vec{\alpha}$-Baire property (is $\vec{\alpha}$-meager) iff $\mathcal{X}$ has the $\mathcal{R}_{\vec{\alpha}}$-Baire property (is $\mathcal{R}_{\vec{\alpha}}$-meager). Therefore, $(\mathcal{R},\mathcal{R}_{\vec{\alpha}},\le,r)$ is a topological Ramsey space.
\end{proof}

\begin{thm}
Assume that $(\mathcal{R},\le, r)$ satisfies \emph{A.1}, \emph{A.2} and \emph{A.4} and for all $s\in\mathcal{AR}$, $\leftidx{^{*}}{s}=s$. Suppose that the $\mathfrak{c}^{+}$-enlarging property holds.   Let 
$$\mathcal{R}_{\vec{\mathcal{U}}}=\{ X\in \mathcal{R} : \forall s\in \mathcal{AR}\restriction X, \ r_{|s|+1}[s,X]\in\mathcal{U}_{s}\}.$$ If for all $\vec{\mathcal{U}}$-trees $T$ there exists $X\in \mathcal{R}_{\vec{\mathcal{U}}}$ such that $\emptyset\not=[st(T),X]\subseteq [T]$, then $(\mathcal{R}, \mathcal{R}_{\vec{\mathcal{U}}}, \le, r)$ is a topological Ramsey space.
\end{thm}
\begin{proof}
By Proposition \ref{abstract U-tree alpha-tree} there exists a sequence $\vec{\alpha}=\left<\alpha_{s} : s\in\mathcal{AR}\right>$ such that for all trees $T$ on $\mathcal{R}$, $T$ is an $\vec{\mathcal{U}}$-tree if and only if $T$ is an $\vec{\alpha}$-tree. This result is simply a restatement of the previous theorem in the setting of $\vec{\mathcal{U}}$-trees instead of $\vec{\alpha}$-trees.
\end{proof}

\begin{example}[The Ellentuck Space] Consider the triple  $([\mathbb{N}]^{\infty},\subseteq, r)$.
Let $\mathcal{U}$ be a non-principal ultrafilter on $\mathbb{N}$. By Lemma \ref{infinite set} there exists a nonstandard hypernatural number $\beta$ such that $\mathcal{U} =\{ X\subseteq \mathbb{N} : \beta\in\leftidx{^{*}}{X}\}$. For each $s\in [\mathbb{N}]^{<\infty}$, let $\mathcal{U}_{s}=\{ \{s\cup\{n\} : n\in X \ \& \max(s)<n\} : X\in \mathcal{U}\}$ and $\vec{\mathcal{U}}=\left<\mathcal{U}_{s}:s\in[\mathbb{N}]^{<\infty}\right>$. Note that for all $s\in [\mathbb{N}]^{<\infty}$, $\mathcal{U}_{s}$ is a non-principal ultrafilter on $\{t\in[\mathbb{N}]^{<\infty} : s\sqsubseteq t \ \& \ |t|=|s|+1\}$. Notice  also that for all $X\subseteq \mathbb{N}$,
\begin{align*}
X\in[\mathbb{N}]^{\infty}_{\vec{\mathcal{U}}} &\iff \forall s\in[X]^{<\infty}, \ s\cup \{\beta\} \in \leftidx{^{*}}{r_{|s|+1}[s,X]},\\
&\iff \emptyset \cup \{\beta\} \in \leftidx{^{*}}{r_{1}[\emptyset,X]},\\
&\iff \beta\in \leftidx{^{*}}{X} \iff X\in \mathcal{U}.
\end{align*}
Thus, by the previous theorem, if for all $\vec{\mathcal{U}}$-trees $T$ there exists $X\in \mathcal{U}$ such that $\emptyset\not=[st(T),X]\subseteq[T]$, then $([\mathbb{N}]^{\infty}, \mathcal{U},\subseteq, r)$ is a topological Ramsey space.
\end{example}

\begin{example}[The Milliken Space] Consider the triple $([\mathrm{FIN}]^{\infty},\le,r)$. If $s\in [\mathrm{FIN}]^{<\infty}$ and $x\in \mathrm{FIN}$ with $\max(s)<\min(x)$ then we let $s^{\frown}x$ denote the element of $[\mathrm{FIN}]^{<\infty}$ obtained by concatenating $x$ to the end of $s$. For each $S\in [\mathrm{FIN}]^{\infty}$, let $FU(S)$ denote the collection of all finite unions of elements of the sequence $S$. Then $S\le T$ if and only if for all $i\in\mathbb{N}$, $s_{i}\in FU(T)$.

Let $\mathcal{U}$ be an ultrafilter on $\mathrm{FIN}$ such that for all $s\in [\mathrm{FIN}]^{<\infty}$, $\{\{s^{\frown}x : x\in X \ \& \ \max(s)<\min(x)\}: X\in\mathcal{U}\}$ is a non-principal ultrafilter on $\{t\in[\mathrm{FIN}]^{<\infty} : s\sqsubseteq t \ \& \ |t|=|s|+1\}$. Lemma \ref{need in milliken example} implies that there exists $\beta\in\leftidx{^{*}}{\mathrm{FIN}}$ such that $\mathcal{U} =\{ X\subseteq \mathrm{FIN} : \beta\in \leftidx{^{*}}{X}\}$. Note that $\min(\beta)$ is a nonstandard hypernatural number because otherwise $\{\{\min(\beta)\}^{\frown}x : x\in X \ \& \ \min(\beta)<\min(x)\}: X\in\mathcal{U}\}$ would not be an ultrafilter on $\{t\in[\mathrm{FIN}]^{<\infty} : \{\min(\beta)\}\sqsubseteq t \ \& \ |t|=2\}$. For all $s\in[\mathrm{FIN}]^{<\infty}$, let $\mathcal{U}_{s} = \{\{s^{\frown}x : x\in X \ \& \ \max(s)<\min(x)\}: X\in\mathcal{U}\}$ and $\vec{\mathcal{U}}=\left<\mathcal{U}_{s} : s\in[\mathrm{FIN}]^{<\infty}\right>$. Notice that for all $S\in [\mathrm{FIN}]^{\infty}$,
\begin{align*}
S\in [\mathrm{FIN}]^{\infty}_{\vec{\mathcal{U}}} &\iff \forall s\in [\mathrm{FIN}]^{<\infty}\restriction S, \ s^{\frown}\beta \in \leftidx{^{*}}{r_{|s|+1}[s, S]},\\ 
&\iff \emptyset^{\frown}\beta \in \leftidx{^{*}}{r_{1}[\emptyset, S]},\\ 
&\iff \beta \in \leftidx{^{*}}{FU(S)} \iff FU(S)\in \mathcal{U}.
\end{align*}
By previous theorem, if for all $\vec{\mathcal{U}}$-trees $T$ there exists $S\in[\mathrm{FIN}]^{\infty}$ such that $FU(S)\in\mathcal{U}$ and $\emptyset\not=[st(T),S]\subseteq [T]$, then $([\mathrm{FIN}]^{\infty}, \{S\in[\mathrm{FIN}]^{\infty}: FU(S)\in\mathcal{U}\}, \le, r)$ is a topological Ramsey space. 

In \cite{MijaresSelective}, Mijares introduced a notion of selective ultrafilter corresponding to a topological Ramsey space satisfying axioms A.1-A.4. Using the theory of forcing Mijares showed that the existence of these selective ultrafilters is consistent with ZFC. In fact, the existence of such selective ultrafilters follows from the continuum hypothesis or Martin's Axiom. For the Ellentuck space and Milliken space theses selective ultrafilters give rise to sequences $\vec{\mathcal{U}}$ such that for all $\vec{\mathcal{U}}$-trees $T$ there exists $X\in \mathcal{R}_{\vec{\mathcal{U}}}$ such that $\emptyset\not=[st(T),X]\subseteq [T]$. In particular, the existence of the topological Ramsey spaces $([\mathrm{FIN}]^{\infty}, \{S\in[\mathrm{FIN}]^{\infty}: FU(S)\in\mathcal{U}\}, \le, r)$ and $([\mathbb{N}]^{\infty}, \mathcal{U},\subseteq, r)$ mentioned in the previous two examples is consistent with ZFC. Thus, their existence is also consistent with the Alpha Theory. 

Mijares in \cite{MijaresSelective}, shows that for any topological Ramsey space $(\mathcal{R},\le,r)$ satisfying A.1-A.4 the forcing notion which Mijares calls almost-reduction forces the existence of a selective ultrafilter $\mathcal{U}$ for $\mathcal{R}$. Moreover, he use a combinatorial forcing argument to show that $(\mathcal{R},\mathcal{U},\le,r)$ forms a topological Ramsey space. When restricted to the Milliken space and Ellentuck space the previous two examples show a different proof of this fact which use the Abstract Ultra-Ellentuck Theorem. However, in the general case it is unknown if the forcing gives rise to a sequence $\vec{\mathcal{U}}$ such that for all $\vec{\mathcal{U}}$-trees $T$ there exists $X\in \mathcal{R}_{\vec{\mathcal{U}}}$ such that $\emptyset\not=[st(T),X]\subseteq [T]$. If it is the case then an argument similar the previous two examples will hold for any topological Ramsey space satisfying A.1-A.4. 

More recently, Di Prisco, Mijares and Nieto in \cite{DMN} have developed an abstract version of the local Ramsey theory from the previous section. They extend the notion of a selective ultrafilter for a Ramsey space to the more general notion of semiselective coideal corresponding to the Ramsey space. These spaces give rise to Ramsey spaces that are not necessarily topological.  They also give a slightly different form of the definition of selective ultrafilter introduced by Mijares in \cite{MijaresSelective}. They show that with the continuum hypothesis, Martin's Axiom or forcing with almost-reduction the existence of such selective ultrafilters is consistent with ZFC. 

In \cite{DT1,DT2,Tr1,Tr2}, another notion of selective ultrafilter for the space is used; in fact, in \cite{DMN} Di Prisco, Mijares and Nieto mention that these definitions fail A.3 of their definition. Trujillo in \cite{Tr2} has shown that there is a topological Ramsey space $\mathcal{R}$ and an ultrafilter that is selective with respect $\mathcal{R}$, with the varying definition, such that $(\mathcal{R},\mathcal{U},\le,r)$ is not a topological Ramsey space (more precisely, he showed that the selective ultrafilter is not Ramsey for the space). This example, with the definition of selective from \cite{DT1,DT2,DMT,Tr1,Tr2}, would imply that the answer to the next question is false. However, when restricted to the Millken space or the Ellentuck space it can be shown that this definition is in fact enough to show that $(\mathcal{R},\mathcal{U},\le,r)$ is a topological Ramsey space.
\begin{question}
Let $(\mathcal{R},\le,r)$ be a topological Ramsey space satisfying A.1-A.4. Suppose that $\mathcal{U}\subseteq \mathcal{R}$ a selective ultrafilter with respect to $\mathcal{R}$ as defined in \cite{DMN}. For each $s\in\mathcal{AR}$, let $\mathcal{U}_{s}$ be the ultrafilter on $\{t\in \mathcal{AR}_{|s|+1}: s\sqsubseteq t\}$ generated by $\{ r_{|s|+1}[s,X] : X\in \mathcal{U}\}$ and $\vec{\mathcal{U}}=\left<\mathcal{U}_{s} :s\in\mathcal{AR}\right>$. Is it the case that for all $\vec{\mathcal{U}}$-trees $T$ there exists $X\in \mathcal{R}_{\vec{\mathcal{U}}}$ such that $\emptyset\not=[st(T),X]\subseteq [T]$?
\end{question}

If the answer to the previous question is yes, then the triple $(\mathcal{R},\mathcal{R}_{\vec{\mathcal{U}}},\le,r)$ will form a topological Ramsey space. However, if the answer is no then there is a topological Ramsey space $(\mathcal{R},\mathcal{R}_{\vec{\mathcal{U}}},\le,r)$ where there exists an $\vec{\mathcal{U}}$-tree $T$ such that for all $X\in \mathcal{R}_{\vec{\mathcal{U}}}$, $\emptyset=[st(T),X]$ or $[st(T),X]\not\subseteq [T]$.
\end{example}

\bibliographystyle{plain}
\bibliography{AlphaRamseyTheory1}

\end{document}